\documentclass[preprint,12pt]{elsarticle}

\usepackage{amsmath, amsthm, amssymb}
\usepackage{mathtools}
\usepackage{hyperref}
\usepackage{graphicx}
\usepackage{tikz}
\usetikzlibrary{decorations.pathreplacing}

\newtheorem{theorem}{Theorem}[section]

\newtheorem{corollary}[theorem]{Corollary}
\newtheorem{remark}[theorem]{Remark}

\newtheorem{assumption}[theorem]{Assumption}

\begin{document}

\begin{frontmatter}

\title{Spectral Analysis and Redistribution Thresholds for Cut-Cell Finite-Volume Methods}

\author[stevens]{Justo E. Karell\corref{cor1}}
\ead{justokarell@gmail.com}
\cortext[cor1]{Corresponding author}
\affiliation[stevens]{
  organization={Department of Computer Science, Stevens Institute of Technology},
  addressline={1 Castle Point Terrace},
  city={Hoboken},
  state={NJ},
  postcode={07030},
  country={USA}
}

\begin{abstract}
In finite-volume methods on embedded-boundary meshes, arbitrarily small cut
cells can create $O(\alpha^{-1})$ coefficients when the time step is chosen
for the regular grid.  For a one-dimensional periodic upwind discretization,
we show that this instability is carried by a discrete eigenmode localized at
the cut cell.  Blending the unstabilized update with the volume-weighted
state obtained by merging the cut cell with its left neighbor gives a spectral
threshold whose leading value, when the cut-cell Courant number $L$ is fixed,
is
\[
    s_0(L)=\frac{L-2}{L-1},\qquad L>2.
\]
We prove that the cell-merging correction becomes aligned with the localized
unstable mode as $\alpha\to0$.  We also prove a block-matrix result showing
that cut-cell rows with coefficients that diverge as $\alpha\to0$ generate a finite cluster of unbounded
eigenvalues whose invariant subspace approaches the subspace spanned by the
cut-cell coordinates.  For the second-order monotonic upstream-centered
scheme for conservation laws (MUSCL) with a minmod limiter, we analyze the branch of the piecewise-linear minmod map containing the
localized mode.  The corresponding
nonlinear eigenvalue problem gives
\[
    s_{\mathrm M}(\alpha,L)
    =
    \frac{L-2}{L-1}
    -
    \frac{L(L-2)}{4(L-1)^2}\alpha
    +O(\alpha^2),
\]
and the second-order strong-stability-preserving Runge--Kutta (SSPRK2) update
reaches unit amplification at the same crossing on that branch.  In the
sine-wave benchmark at $E_{1,h}(T)=10^{-3}$, this MUSCL value requires an
estimated $241$ cells versus $286$ for the 2022 $\alpha$--$\beta$ weighting
with neighborhood reconstruction slopes set to zero, $334$ for the tested
specialization of the 2024 provably stable weighted method, and $340$ for full
two-cell merging when $\alpha=0.15$.  For $\alpha=0.10$, the corresponding
counts are $277$, $292$, $323$, and $326$.  Across five interpolated error
levels, the reduction relative to the tested 2024 specialization is
$25.0\%$--$28.5\%$ for $\alpha=0.15$ and $12.9\%$--$14.3\%$ for
$\alpha=0.10$.  On a separate two-dimensional $45^\circ$ embedded-boundary
channel, the smallest blending parameter for which the assembled
first-order matrix has spectral radius at most one is $0.363$--$0.879$, rather
than $1$.  The one-dimensional formula is nearly exact for the smallest
triangles and differs more for the larger cut cells.  These results show that
full redistribution is not always necessary to control the unstable cut-cell
mode, and that using a smaller blending parameter can reduce numerical
error in the tested problems.
\end{abstract}


\end{frontmatter}


\footnotetext{Abbreviations used in the manuscript: CFL, Courant--Friedrichs--Lewy;
SRD, state redistribution; MUSCL, monotonic upstream-centered scheme for
conservation laws; SSPRK2, second-order strong-stability-preserving Runge--Kutta.}

\section{Introduction}
\label{sec:intro}

Embedded-boundary and cut-cell methods represent complex physical boundaries
while retaining a simple background mesh.  Their main explicit-time-stepping
difficulty is geometric.  If a boundary leaves only a fraction
$\alpha\in(0,1]$ of a background control volume, then a conservative
finite-volume update divides the net numerical flux by this reduced volume.
A time step that satisfies the background Courant--Friedrichs--Lewy (CFL) condition can therefore produce
$O(\alpha^{-1})$ coefficients in the cut-cell row of the update matrix.

Several methods remove this small-cell restriction.  Flux redistribution
\cite{ColellGravesKeenModiano2006}, h-box methods
\cite{BergerHelzelLeVeque2003}, explicit-implicit updates
\cite{MayBerger2017,MayLaakmann2024}, and domain-of-dependence stabilization
\cite{EngwerMayNussingStreitburger2020,MayStreitburger2022,Giuliani2022,PetriBirkeEngwerRanocha2026}
modify the local update or its effective domain of dependence.  State
redistribution (SRD) instead forms a provisional finite-volume update, merges
small cells with neighboring cells into larger neighborhoods, and
redistributes the merged states while preserving conservation
\cite{BergerGiuliani2021}.  Weighted SRD reduces dissipation through
geometry-dependent weights \cite{GiulianiEtAl2022}; Berger and Giuliani
\cite{BergerGiuliani2024} further construct a weighted SRD scheme with
monotonicity, total-variation-diminishing, and GKS-stability results in the
configurations covered by their analysis.  Update-Magnitude SRD (UM-SRD)
blends redistribution according to the magnitude of the provisional update
\cite{Karell2026}.  The scalar blend studied here is distinct from these
weight constructions: the redistribution operator is held fixed and only its
overall blend with the base update is varied.

A closely related stability analysis is Berger's study of cut-cell merging
\cite{Berger2015}.  For one-dimensional linear advection with first-order
upwinding, that work uses GKS stability theory to quantify the admissible time
step after cell merging and shows that the usual cut-cell CFL restriction can
be overly conservative.  Berger and Giuliani \cite{BergerGiuliani2024}, by
contrast, construct geometry-dependent weighted SRD coefficients and prove
stability properties of that weighted algorithm.  The question here is
different from both: the base update and the redistribution operator are held
fixed, only the scalar blend between them is varied, and the blend is selected
by following the localized eigenvalue of the full discrete operator to the
unit circle.  The resulting quantity is therefore a redistribution threshold for the blending parameter, not a time-step bound or a new geometric weighting
formula.

Let
\[
    \overline{\mathbb D}:=\{z\in\mathbb C:|z|\le1\}
\]
denote the closed unit disk.  We consider the following threshold problem:
given a finite-volume operator with a localized eigenvalue outside
$\overline{\mathbb D}$ and a specified redistribution operator, determine the
scalar blend at which that eigenvalue re-enters $\overline{\mathbb D}$.  We
analyze this problem for positive-speed scalar
linear advection on a one-dimensional periodic mesh of $N$ cells with one cut
cell, discretized by first-order upwind with cell merging as the redistribution operator.  Let
$N$ denote the total number of cells and let $\lambda$ denote the background
CFL number.  Define the local cut-cell CFL number by
\[
    L=\lambda_\alpha=\frac{\lambda}{\alpha}.
\]  The unstabilized operator has a localized
eigenvalue near $1-L$, so $L>2$ places that eigenvalue outside
$\overline{\mathbb D}$.  For a square matrix $M$, write $\sigma(M)$ for its
spectrum and
\[
    \rho(M):=\max_{\mu\in\sigma(M)}|\mu|.
\]
Throughout the paper, we use the condition $\rho(A_s)\le1$ only as a redistribution condition; because $A_s$ is generally nonnormal, this condition is not identified with power boundedness.
Because $A_s$ is generally nonnormal, power boundedness and finite transient
amplification remain separate questions.  Let $A$ denote the base update matrix and
$A_{\mathcal R}$ the fully redistributed update matrix.  We blend them as
\[
    A_s=(1-s)A+sA_{\mathcal R},\qquad 0\le s\le1,
\]
where $s$ is the blending parameter $s$: $s=0$ gives the base update and
$s=1$ gives full cell merging.  We then derive the full $N$-cell
characteristic equation.  For fixed $L>2$ and
$N\ge4$, the localized eigenvalue crosses the value $-1$ at
\[
 s_{-1}^{(N)}(\alpha,L)
 =
 \frac{L-2}{L-1}
 -
 \frac{L(L-2)}{2(L-1)^2}\alpha^2
 +O(\alpha^3).
\]
All other nonconstant eigenvalues are inside the unit disk near this crossing
for sufficiently small $\alpha$.  Consequently,
\[
    s_0(L)=\frac{L-2}{L-1}
\]
is asymptotically sufficient and differs from the exact local redistribution
threshold by $O(\alpha^2)$.  The contribution is a derivation of the
blending parameter from the localized unstable branch of the discrete
spectrum.

The unstable eigenvector converges to the cut-cell coordinate vector as
$\alpha\to0$, so the unstable mode is concentrated at the cut cell.
The two-cell-merging operation changes the state in an asymptotically aligned direction.
Consequently, varying $s$ directly changes the amplification factor of the
localized branch used to define the redistribution threshold.
Eigenvalue spectra of embedded-boundary discretizations have previously been
used as a stability diagnostic \cite{DevendranGravesJohansenLigocki2017}; here
the characteristic equation identifies a localized amplification mode and its
spectral crossing analytically.

Two limits are used and are kept separate throughout.  For the singular
eigenvalue and invariant-subspace analysis, $\lambda$ is fixed while
$\alpha\to0$, so $L=\lambda/\alpha\to\infty$.  For the spectral-threshold
expansion, $L$ is fixed while $\alpha\to0$, so $\lambda=\alpha L\to0$.
The higher-order analysis later derives the corresponding localized
MUSCL--minmod crossing directly.  Its finite-$\alpha$ numerical benchmark
uses a fixed background CFL and tests both the leading value $s_0(L)$ and the
finite-$\alpha$ MUSCL value $s_{\mathrm M}(\alpha,L)$.

At fixed regular-grid CFL $\lambda$, the local CFL is $L=\lambda/\alpha$, so the leading
prescription becomes
\[
    s_0(\lambda/\alpha)
    =\frac{\lambda-2\alpha}{\lambda-\alpha}
    =1-\frac{\alpha}{\lambda-\alpha}
    \longrightarrow 1
    \qquad (\alpha\to0).
\]
Thus the difference $1-s_0$ from full cell merging is largest for moderately
restrictive cut cells.  For arbitrarily small cut cells at a fixed regular-grid
time step, $s_0$ approaches $1$.

We next identify which parts of the model calculation extend beyond one cut
cell.  After placing the cut-cell unknowns first, a class of update matrices has
a finite-dimensional block whose entries grow as $\alpha\to0$.  We prove that
the leading coefficient matrix of this block determines the unbounded
eigenvalues and that the corresponding invariant subspace approaches the
subspace spanned by the cut-cell coordinates.  This localization result does
not depend on the particular scalar threshold formula derived for the upwind
model.

The main contributions are:
\begin{enumerate}
    \item an explicit characteristic-polynomial analysis of the one-cut-cell
    periodic upwind operator, including the unbounded eigenvalue and
    eigenvector localization;
    \item a definition of cell merging on the same full operator and a proof
    that its correction becomes asymptotically aligned with the localized
    unstable direction in the volume-weighted inner product;
    \item a full-$N$ derivation of the local eigenvalue crossing and the
    leading-order prescription $s_0(L)=(L-2)/(L-1)$ for $L>2$;
    \item a block-matrix theorem relating the matrix $\Gamma$ to the unbounded spectrum and localized invariant subspace, with
    multiple-cut-cell, multidimensional, and finite-volume consequences;
    \item a calculation of $\mathcal E_s(v)=\nu v$ in the localized MUSCL--minmod
    piecewise-linear region, together with an SSPRK2 corollary showing the same crossing;
    \item a two-dimensional rotated-channel experiment in which the smallest
    $s$ satisfying $\rho(A_s)\le1$ is below full cell merging, with the
    one-dimensional prediction nearly exact for the smallest tested cut cells;
    and
    \item higher-order numerical evidence that the blending parameter from
    the MUSCL analysis preserves the observed convergence order, lowers
    smooth-solution error, and reaches prescribed error levels with smaller
    interpolated total cell counts than the tested stronger redistributions.
\end{enumerate}

The first-order threshold theorem applies to the periodic upwind model with
cell merging, and Section~\ref{sec:muscl-theory} derives the analogous
localized crossing for the MUSCL--minmod stage.  The multidimensional result
establishes leading-order localization at the cut cell, while the two-dimensional
experiment in Section~\ref{sec:2d-experiment} computes the threshold for the assembled
first-order spectrum directly.  Nonlinear conservation laws, systems, adjacent cut cells, and
multidimensional weighted-SRD neighborhoods still require separate analysis
of the corresponding discrete operator and of how redistribution changes its
unstable eigenvalues.

The paper follows the argument in that order.  Section~\ref{sec:1d} defines
the model problem and derives the localized unstable mode.
Section~\ref{sec:stabilization} defines cell merging, proves its alignment
with the mode, and derives the local eigenvalue-crossing threshold for the
full periodic operator.
Section~\ref{sec:singular} develops the block-matrix analysis.
Sections~\ref{sec:multi}--\ref{sec:verify} give multiple-cut-cell,
multidimensional, and finite-volume consequences.  Section~\ref{sec:muscl-theory}
derives the MUSCL--minmod localized crossing and its SSPRK2 consequence.
Section~\ref{sec:numerics} then visualizes the one-dimensional localized eigenvalue behavior, computes the redistribution threshold of an assembled two-dimensional cut-cell operator, and
measures the higher-order accuracy consequences of partial redistribution.

\section{One-Dimensional One-Cut-Cell Model}
\label{sec:1d}

\subsection{Model problem and update operator}

Consider the scalar advection equation
\begin{equation}
    u_t+a u_x=0,\qquad a>0,
    \label{eq:advection}
\end{equation}
on a periodic one-dimensional domain discretized into $N$ cells indexed
periodically by $i=0,\ldots,N-1$.  The periodic setting isolates the
small-volume singularity from physical boundary-condition effects; it is a
model operator for the small-cell instability rather than a geometrically literal
one-dimensional embedded boundary.  Let $h$ denote the width of a regular
background cell, let $\Delta t$ be the time step, and define the background
CFL number
\[
    \lambda=\frac{a\Delta t}{h},\qquad 0<\lambda\le1.
\]
All cells have length $h$ except one cut cell $c$, whose length is
$\alpha h$ with $\alpha\in(0,1)$.  Let $U_i^n$ denote the cell average in
cell $i$ at time $t^n$, and let $U^n=(U_0^n,\ldots,U_{N-1}^n)^\top$.
Following the notation of Berger and Giuliani for SRD, we denote the
\emph{provisional} base-scheme update by $\widehat U$.  The positive-speed
first-order upwind step is
\[
 \widehat U_i
 =
 U_i^n-\frac{a\Delta t}{|C_i|}
 \left(U_i^n-U_{i-1}^n\right),
 \qquad
 |C_i|=\begin{cases}\alpha h,&i=c,\\ h,&i\ne c.\end{cases}
\]
Thus $\widehat U=AU^n$, where the provisional update matrix $A$ is defined by
\begin{equation}
    (AU)_i=(1-\lambda_i)U_i+\lambda_i U_{i-1},
    \qquad
    \lambda_i=
    \begin{cases}
        \lambda_\alpha:=\lambda/\alpha, & i=c,\\
        \lambda, & i\ne c,
    \end{cases}
    \label{eq:A}
\end{equation}
with periodic indexing.  Let $e_i\in\mathbb R^N$ denote the $i$th canonical
coordinate vector.  The only singular row as $\alpha\to0$ is the cut-cell
row.  If $A_{\mathrm{bg}}$ denotes the uniform periodic upwind matrix at CFL $\lambda$,
then
\begin{equation}
    A
    =
    A_{\mathrm{bg}}+(\lambda_\alpha-\lambda)e_c(e_{c-1}-e_c)^\top.
    \label{eq:rankone}
\end{equation}

Equation~\eqref{eq:A} is the algebraic form of the small-cell problem.  The
same flux difference that is divided by $h$ on a regular cell is divided by
$\alpha h$ on the cut cell.  Thus the local CFL number is
$\lambda_\alpha=\lambda/\alpha$.  The analysis below determines how this
single singular row changes the spectrum.

\subsection{Characteristic polynomial and the localized root}

By periodic translation we may set $c=0$.  Let $Av=\mu v$.  On every regular
row $i=1,\ldots,N-1$,
\[
    (\mu-1+\lambda)v_i=\lambda v_{i-1}.
\]
For $\mu\ne1-\lambda$, define
\begin{equation}
    q=\frac{\lambda}{\mu-1+\lambda}.
    \label{eq:qdef}
\end{equation}
Then
\begin{equation}
    v_i=q^i v_0,\qquad i=1,\ldots,N-1,
    \label{eq:rec}
\end{equation}
and
\begin{equation}
    \mu(q)=1-\lambda+\frac{\lambda}{q}.
    \label{eq:muq-base}
\end{equation}
Substituting $v_{N-1}=q^{N-1}v_0$ into the cut-cell row gives
\begin{equation}
    P_\alpha(q):=q^N-(1-\alpha)q-\alpha=0.
    \label{eq:poly}
\end{equation}
Conversely, every nonzero root of~\eqref{eq:poly} gives an eigenvalue through
\eqref{eq:muq-base}.  Since $P_\alpha(0)=-\alpha$, all roots are nonzero for
$\alpha>0$.

The localized root lies near the zero of the linear part of $P_\alpha$.  Set
\[
    \widetilde q=-\frac{\alpha}{1-\alpha}.
\]
Writing $q=\widetilde q+\eta$ and using
$(1-\alpha)\widetilde q+\alpha=0$ gives
\[
    \eta=\Phi_\alpha(\eta)
    :=\frac{(\widetilde q+\eta)^N}{1-\alpha}.
\]
On the disk
\[
    |\eta|\le
    r_\alpha:=\frac{2|\widetilde q|^N}{1-\alpha},
\]
we have $r_\alpha=O(\alpha^N)=o(|\widetilde q|)$.  Therefore, for sufficiently
small $\alpha$, $\Phi_\alpha$ maps the disk into itself and
\[
    \sup_{|\eta|\le r_\alpha}|\Phi_\alpha'(\eta)|
    =
    O(\alpha^{N-1})<1.
\]
The contraction mapping theorem gives a unique root in this disk:
\begin{equation}
    q_u=-\frac{\alpha}{1-\alpha}+O(\alpha^N).
    \label{eq:qu}
\end{equation}

\begin{theorem}[Unique unbounded eigenvalue]
\label{thm:eigenvalue}
Fix $N\ge3$ and $\lambda\in(0,1]$.  As $\alpha\to0$, the matrix $A$ has a
unique unbounded eigenvalue
\begin{equation}
    \mu_u
    =
    1-\frac{\lambda}{\alpha}
    +O(\alpha^{N-2}),
    \label{eq:muu}
\end{equation}
while all other eigenvalues remain bounded.
\end{theorem}

\begin{proof}
The root~\eqref{eq:qu} gives
\[
    \frac1{q_u}
    =
    -\frac{1-\alpha}{\alpha}
    +O(\alpha^{N-2}),
\]
and substitution into~\eqref{eq:muq-base} yields~\eqref{eq:muu}.

It remains to show that no second root approaches zero.  Choose a fixed
$C>1$.  For sufficiently small $\alpha$, the circle $|q|=C\alpha$ contains
$\widetilde q$.  Write
\[
    P_\alpha(q)=g_\alpha(q)+q^N,\qquad
    g_\alpha(q)=-(1-\alpha)q-\alpha.
\]
On $|q|=C\alpha$,
\[
 |g_\alpha(q)|
 \ge
 \alpha\bigl((1-\alpha)C-1\bigr),
 \qquad
 |q^N|=C^N\alpha^N.
\]
For sufficiently small $\alpha$, $|q^N|<|g_\alpha(q)|$ on the circle.
Rouch\'e's theorem therefore gives exactly one root inside
$|q|<C\alpha$.

For the remaining roots,
\[
 P_\alpha(q)\longrightarrow q(q^{N-1}-1)
\]
coefficientwise.  The $N-1$ nonzero limiting roots are simple and lie on the
unit circle, so the remaining roots of $P_\alpha$ stay bounded away from zero.
Equation~\eqref{eq:muq-base} then shows that their eigenvalues remain bounded.
\end{proof}

\begin{corollary}[Limiting redistribution boundary]
\label{cor:threshold}
Fix $N\ge3$ and let $L=\lambda/\alpha$ remain in a compact subset of
$(0,\infty)$ as $\alpha\to0$, so that $\lambda=\alpha L$.  The localized
branch satisfies
\begin{equation}
    \mu_u=1-L+O(\alpha^{N-1}).
    \label{eq:muu-fixedL}
\end{equation}
The constant branch is exactly $\mu=1$, and every other branch satisfies
\begin{equation}
    \mu_k
    =1+\alpha L(q_k^{-1}-1)+O(\alpha^2),
    \qquad q_k^{N-1}=1,\quad q_k\ne1.
    \label{eq:regular-fixedL}
\end{equation}
Consequently, for every $\varepsilon>0$, the full base operator is spectrally stable for $0<L\le2-\varepsilon$ and spectrally unstable for
$L\ge2+\varepsilon$, once $\alpha$ is sufficiently small.  Thus $L=2$ is the
limiting redistribution boundary.  No finite-$\alpha$ stability claim is made at
$L=2$ itself.
\end{corollary}

\begin{proof}
The root expansion~\eqref{eq:qu} is independent of $\lambda$.  Substituting
$\lambda=\alpha L$ into~\eqref{eq:muq-base} gives
\[
 \mu_u
 =1-\alpha L+\alpha L\left[-\frac{1-\alpha}{\alpha}
      +O(\alpha^{N-2})\right]
 =1-L+O(\alpha^{N-1}),
\]
which proves~\eqref{eq:muu-fixedL}.  The root $q=1$ of~\eqref{eq:poly} is
exact and gives the conservation eigenvalue $\mu=1$.  The other roots converge
to the nontrivial $(N-1)$st roots of unity.  For such a root
$q_k=e^{i\varphi_k}$, analytic perturbation and~\eqref{eq:muq-base} give
\eqref{eq:regular-fixedL}, hence
\[
 |\mu_k|^2
 =1+2\alpha L(\cos\varphi_k-1)+O(\alpha^2)<1
\]
for sufficiently small $\alpha$.  If $L\le2-\varepsilon$, then
$|1-L|\le1-\varepsilon$ and~\eqref{eq:muu-fixedL} places the localized branch
strictly inside the unit disk.  If $L\ge2+\varepsilon$, then
$|1-L|\ge1+\varepsilon$ and the localized branch lies outside it.  This proves
the stated full-operator result away from the boundary $L=2$.
\end{proof}

The value $2$ should not be interpreted as a universal cut-cell CFL condition.
It is the limiting boundary for the particular periodic first-order upwind
operator~\eqref{eq:A}.  A finite-$\alpha$ version of the localized crossing can
be written directly in terms of the localized root $q_u(\alpha)$ of
\eqref{eq:poly}.  Since that root is independent of $L$, setting
$\mu_u=-1$ in \eqref{eq:muq-base} gives
\begin{equation}
    L_*(\alpha,N)
    =
    \frac{2q_u(\alpha)}{\alpha\,[q_u(\alpha)-1]}
    =
    2+O(\alpha^{N-1}).
    \label{eq:Lstar-finite-alpha}
\end{equation}
Thus the finite-$\alpha$ displacement of the localized crossing from $L=2$
is extremely small when $N$ is large.  This formula identifies the crossing
of the localized branch; the full spectral statement at finite $\alpha$
still requires the remaining branches to lie inside the unit disk, as in the
proof of Corollary~\ref{cor:threshold}.

\begin{theorem}[Eigenvector localization]
\label{thm:loc1d}
Normalize the eigenvector associated with $\mu_u$ by $(v_u)_c=1$.  Let
$d_+(c,i)\in\{0,\ldots,N-1\}$ be the number of forward upwind-recurrence steps
from $c$ to $i$ around the periodic mesh.  Then
\begin{equation}
    |(v_u)_i|
    =
    O\!\left(\alpha^{d_+(c,i)}\right),
    \qquad
    \|v_u-e_c\|_2=O(\alpha).
    \label{eq:loc1d}
\end{equation}
\end{theorem}

\begin{proof}
For the localized root, $q_u=O(\alpha)$.  Equation~\eqref{eq:rec} therefore
gives
\[
    (v_u)_i=q_u^{d_+(c,i)}
\]
under the normalization $(v_u)_c=1$.  The component at the nearest forward regular cell is $O(\alpha)$, and
subsequent components are higher powers of $\alpha$.
Summing their squares gives~\eqref{eq:loc1d}.
\end{proof}

Theorem~\ref{thm:loc1d} shows that the unstable eigenvector is concentrated at
the cut cell: each forward recurrence step introduces another factor
$O(\alpha)$.  This localization explains why a redistribution supported on the
cut-cell neighborhood can alter the localized cut-cell eigenvalue efficiently.

\section{State Redistribution and the redistribution Threshold}
\label{sec:stabilization}

\subsection{Two-cell merging and the correction direction}
\label{sec:align}

We first define the cell-merging operator used throughout the analysis.
Following Berger--Giuliani notation, $\widehat U$ denotes a provisional
base-scheme state and $\widehat Q$ a merged neighborhood state.  In the
operator analysis we suppress the time superscript: for any input vector
$U\in\mathbb R^N$, let $\widehat U=AU$.  We merge the cut cell with its
left neighbor, using the neighborhood
\[
    \mathcal N_c=\{c-1,c\}
\]
and forms the volume-weighted neighborhood state
\begin{equation}
    \widehat Q_c
    =
    \frac{\widehat U_{c-1}+\alpha\widehat U_c}{1+\alpha}.
    \label{eq:Qhat}
\end{equation}
In the present nonoverlapping two-cell setting, full redistribution coincides
with cell merging: both cells in the neighborhood are replaced by the same
volume-weighted average.  We denote the resulting one-step operator by
$A_{\mathcal R}$ and define it by
\begin{equation}
 (A_{\mathcal R}U)_i=
 \begin{cases}
     \widehat Q_c, & i=c-1\ \text{or}\ i=c,\\
     \widehat U_i, & \text{otherwise}.
 \end{cases}
 \label{eq:Asrd-def}
\end{equation}
Because~\eqref{eq:Qhat} is weighted by the physical cell lengths, the local
mass $h\widehat U_{c-1}+\alpha h\widehat U_c$ is preserved.  Because standard SRD retains overlapping neighborhood
contributions, we refer to $s=1$ here as full cell merging rather than ``full
SRD.

Figure~\ref{fig:model-srd} summarizes the geometry and the cell-merging step used in the analysis.  It is a definition schematic rather
than a numerical result.

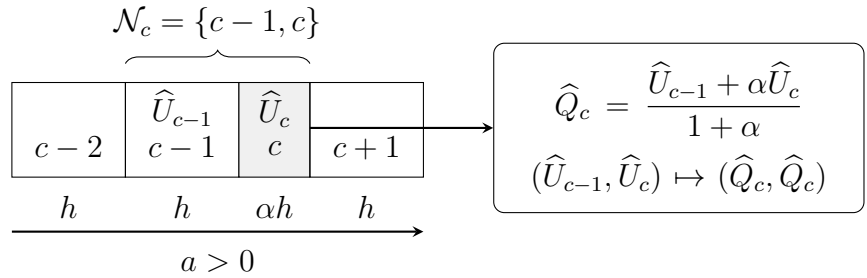
\begin{figure}[ht]
\centering
\begin{tikzpicture}[x=1.25cm,y=1cm,>=stealth]
    % Schematic geometry, not drawn to physical scale.
    \draw (0,0) rectangle (1.2,1.25);
    \draw (1.2,0) rectangle (2.4,1.25);
    \draw[fill=gray!12] (2.4,0) rectangle (3.15,1.25);
    \draw (3.15,0) rectangle (4.35,1.25);

    \node at (0.6,0.38) {$c-2$};
    \node at (1.8,0.38) {$c-1$};
    \node at (2.775,0.38) {$c$};
    \node at (3.75,0.38) {$c+1$};

    \node at (1.8,0.90) {$\widehat U_{c-1}$};
    \node at (2.775,0.90) {$\widehat U_c$};

    \node[below=3pt] at (0.6,0) {$h$};
    \node[below=3pt] at (1.8,0) {$h$};
    \node[below=3pt] at (2.775,0) {$\alpha h$};
    \node[below=3pt] at (3.75,0) {$h$};

    \draw[decorate,decoration={brace,amplitude=5pt}]
        (1.2,1.50) -- (3.15,1.50);
    \node[above=5pt] at (2.175,1.50)
        {$\mathcal N_c=\{c-1,c\}$};

    \node[draw,rounded corners,align=center,text width=4.4cm,inner sep=7pt]
        (merged) at (7.05,0.65)
    {
        $\displaystyle
        \widehat Q_c=
        \frac{\widehat U_{c-1}+\alpha\widehat U_c}{1+\alpha}$
        \\[5pt]
        $(\widehat U_{c-1},\widehat U_c)\mapsto(\widehat Q_c,\widehat Q_c)$
    };
    \draw[->,thick] (3.15,0.65) -- (merged.west);

    \draw[->,thick] (0,-0.72) -- (4.35,-0.72);
    \node[below=3pt] at (2.175,-0.72)
        {$a>0$};
\end{tikzpicture}
\caption{One-dimensional cut-cell geometry and cell merging.
The diagram is schematic and is not drawn to scale.  The cut cell $c$ has
physical width $\alpha h$.  After the base finite-volume step, the provisional
cell averages $\widehat U_{c-1}$ and $\widehat U_c$ are merged over the
neighborhood $\mathcal N_c=\{c-1,c\}$ to form the volume-weighted state
$\widehat Q_c$, which replaces both provisional values.}
\label{fig:model-srd}
\end{figure}

The correction made by full cell merging is
\[
    D(U):=A_{\mathcal R}U-AU.
\]
From~\eqref{eq:Qhat},
\[
 D_{c-1}
 =
 \frac{\alpha(\widehat U_c-\widehat U_{c-1})}{1+\alpha},
 \qquad
 D_c
 =
 -\frac{\widehat U_c-\widehat U_{c-1}}{1+\alpha},
\]
so
\begin{equation}
    D(U)=\Delta(U)\bigl(\alpha e_{c-1}-e_c\bigr),
    \qquad
    \Delta(U)=\frac{\widehat U_c-\widehat U_{c-1}}{1+\alpha}.
    \label{eq:D}
\end{equation}
For the upwind operator, define the two neighboring jumps
\[
    \delta_{c-1}=U_{c-1}-U_{c-2},
    \qquad
    \delta_c=U_c-U_{c-1}.
\]
Then
\begin{equation}
    \Delta(U)
    =
    \frac{\lambda\,\delta_{c-1}
    +(1-\lambda_\alpha)\delta_c}{1+\alpha}.
    \label{eq:Delta-jumps}
\end{equation}

We use the volume-weighted inner product and its induced vector norm:
\begin{equation}
    \langle x,y\rangle_\alpha
    =
    \alpha x_c\overline{y_c}
    +
    \sum_{i\ne c}x_i\overline{y_i},
    \qquad
    \|x\|_\alpha^2=\langle x,x\rangle_\alpha.
    \label{eq:ip}
\end{equation}
For nonzero real vectors $x$ and $y$, define their volume-weighted alignment
by
\begin{equation}
    \mathcal A_\alpha(x,y)
    :=
    \frac{|\langle x,y\rangle_\alpha|}
         {\|x\|_\alpha\|y\|_\alpha}
    \in[0,1].
    \label{eq:alignment-def}
\end{equation}
For sufficiently small $\alpha$, the localized root $q_u$ is real: the
polynomial $P_\alpha$ has real coefficients, and uniqueness of the root in the
contraction disk forces it to equal its complex conjugate.  We therefore take
$v_u\in\mathbb R^N$.

\begin{theorem}[Alignment of the cell-merging correction with the unstable mode]
\label{thm:align}
Let $v_u$ be normalized by $(v_u)_c=1$.  For every real state
$U\in\mathbb R^N$ for which $\Delta(U)\ne0$,
\begin{equation}
    1-\mathcal A_\alpha(D(U),v_u)=O(\alpha)
    \qquad (\alpha\to0).
    \label{eq:align}
\end{equation}
\end{theorem}

\begin{proof}
By~\eqref{eq:D},
\[
 D_c=-\Delta,\qquad D_{c-1}=\alpha\Delta.
\]
With $c=0$, Theorem~\ref{thm:loc1d} gives
$(v_u)_{c-1}=q_u^{N-1}=O(\alpha^{N-1})$ and
$\|v_u-e_c\|_2=O(\alpha)$.  Therefore
\begin{align*}
 \langle D,v_u\rangle_\alpha
 &=
 \alpha\Delta\,(v_u)_{c-1}
 -\alpha\Delta\,(v_u)_c\\
 &=
 -\alpha\Delta+O(\alpha^N\Delta),
\\
 \|D\|_\alpha^2
 &=
 \alpha^2|\Delta|^2+\alpha|\Delta|^2
 =
 \alpha(1+\alpha)|\Delta|^2,
\\
 \|v_u\|_\alpha^2
 &=
 \alpha+O(\alpha^2).
\end{align*}
Hence
\[
 \frac{\langle D,v_u\rangle_\alpha}
      {\|D\|_\alpha\|v_u\|_\alpha}
 =-\frac{\Delta}{|\Delta|}+O(\alpha),
\]
and taking the absolute value gives~\eqref{eq:align}.
\end{proof}

The theorem determines only the asymptotic direction of the SRD correction:
the correction and the localized eigenvector approach the same one-dimensional
subspace in the volume-weighted geometry.  Eigenvalue motion and spectral
stability are determined next from the blended operator itself.

\subsection{Local redistribution threshold}
\label{sec:criterion}

Define the blended update operator
\begin{equation}
    A_s=(1-s)A+sA_{\mathcal R},
    \qquad 0\le s\le1.
    \label{eq:As}
\end{equation}
Thus $s=0$ is the base upwind update and $s=1$ is full cell merging.

We now use the case with fixed cut-cell CFL number.  Fix
\[
    L:=\lambda_\alpha>0,
    \qquad
    \lambda=\alpha L,
\]
and let $\alpha\to0$ with $L$ fixed.  Set $c=0$.  For an eigenvector
$A_sv=\mu v$, the rows $i=1,\ldots,N-2$ are unchanged by SRD and satisfy
\begin{equation}
    v_i=q^i v_0,
    \qquad
    \mu(q)=1-\alpha L+\frac{\alpha L}{q}.
    \label{eq:muq}
\end{equation}
Normalize $v_0=1$ and write $w=v_{N-1}$.  Applying the base operator to the
eigenvector, define the provisional vector $\widehat v:=Av$.  The two
provisional components entering the merge are
\begin{equation}
    \widehat v_{N-1}
    =
    \alpha L q^{N-2}+(1-\alpha L)w,
    \qquad
    \widehat v_0=Lw+1-L.
    \label{eq:vhat-two}
\end{equation}
The corresponding merged neighborhood value is
\[
    \widehat Q_0
    =
    \frac{\widehat v_{N-1}+\alpha\widehat v_0}{1+\alpha}.
\]
The two modified eigenvalue equations are therefore
\[
    \mu w=(1-s)\widehat v_{N-1}+s\widehat Q_0,
    \qquad
    \mu=(1-s)\widehat v_0+s\widehat Q_0.
\]
Eliminating $w$ gives
\begin{align}
G_N(q;\alpha,L,s)
={}&s q\Bigl[
L\alpha^2(q-1)
+L\alpha(2q^{N-1}-q^{N-2}-1)
+L(q^{N-1}-1) \notag\\
&\hspace{2.8cm}
+\alpha(1-q)+(1-q^{N-1})
\Bigr] \notag\\
&-L(1+\alpha)\bigl[\alpha(q-1)+q^N-q\bigr].
\label{eq:GN}
\end{align}
For $\alpha>0$, the $N$ roots of $G_N=0$ correspond to the $N$ eigenvalues
of $A_s$ through~\eqref{eq:muq}.  At $\alpha=0$,
\begin{equation}
    G_N(q;0,L,s)
    =
    q(1-q^{N-1})\bigl[L-(L-1)s\bigr].
    \label{eq:GN0}
\end{equation}
Thus one branch emerges from $q=0$ and the other $N-1$ branches emerge from
the $(N-1)$st roots of unity.

\begin{theorem}[Local redistribution threshold]
\label{thm:criterion}
Fix $N\ge4$ and $L>2$.  Define
\[
    s_0(L)=\frac{L-2}{L-1},
    \qquad
    q_-(\alpha,L)=\frac{\alpha L}{\alpha L-2},
    \qquad
    r_-=q_-^{N-2}.
\]
For all sufficiently small $\alpha>0$, the localized eigenvalue branch of
$A_s$ equals $-1$ at a unique value of $s$ in a neighborhood of $s_0(L)$,
namely
\begin{equation}
 s_{-1}^{(N)}(\alpha,L)
 =
 \frac{(1+\alpha)
 \left[L^2\alpha r_- -L^2\alpha+2L\alpha+2L-4\right]}
 {L^2\alpha^2 r_- -L^2\alpha^2
 +L^2\alpha r_- -L^2\alpha
 +L\alpha r_-+2L\alpha^2+3L\alpha+2L-2\alpha-2}.
 \label{eq:sNexact}
\end{equation}
As $\alpha\to0$ with $L$ fixed,
\begin{equation}
 s_{-1}^{(N)}(\alpha,L)
 =
 \frac{L-2}{L-1}
 -
 \frac{L(L-2)}{2(L-1)^2}\alpha^2
 +O(\alpha^3).
 \label{eq:sNexp}
\end{equation}
Moreover, there exist $\alpha_0>0$ and a neighborhood $\mathcal I$ of
$s_0(L)$ such that for $0<\alpha<\alpha_0$ and $s\in\mathcal I$,
\begin{equation}
    \rho(A_s)\le1
    \quad\Longleftrightarrow\quad
    s\ge s_{-1}^{(N)}(\alpha,L).
    \label{eq:local-equivalence}
\end{equation}
\end{theorem}

\begin{proof}
Setting $\mu=-1$ in~\eqref{eq:muq} gives
\[
    q=q_-(\alpha,L)=\frac{\alpha L}{\alpha L-2}=O(\alpha).
\]
Substitution into~\eqref{eq:GN} yields an equation affine in $s$; solving it
gives~\eqref{eq:sNexact}.  Because $N\ge4$,
$r_-=O(\alpha^{N-2})=O(\alpha^2)$, and expansion of the rational expression
gives~\eqref{eq:sNexp}.  In particular, the coefficient of $\alpha$ vanishes.

It remains to show that this crossing is the local redistribution boundary
of the full periodic matrix.  At $s=s_0(L)$, equation~\eqref{eq:GN0} has one simple root at
$q=0$ and $N-1$ simple roots
\[
    q_k=e^{2\pi i k/(N-1)},\qquad k=0,\ldots,N-2.
\]
The root $q_0=1$ persists exactly for every $\alpha$ and $s$ because
$A_s\mathbf 1=\mathbf 1$, where $\mathbf 1=(1,\ldots,1)^\top$; it is the
conservation eigenvalue $\mu=1$.

For each $q_k\ne1$, the implicit-function theorem gives a nearby root
$q_k(\alpha,s)=q_k+O(\alpha)$, uniformly for $s$ in a sufficiently small
neighborhood of $s_0$.  Using~\eqref{eq:muq},
\[
    \mu_k
    =
    1+\alpha L(q_k^{-1}-1)+O(\alpha^2),
\]
and therefore, with $q_k=e^{i\varphi_k}$,
\[
    |\mu_k|^2
    =
    1+2\alpha L(\cos\varphi_k-1)+O(\alpha^2)<1
\]
for sufficiently small $\alpha$.  Thus every nonconstant regular branch is
strictly inside the unit disk near the crossing.

For the branch issuing from $q=0$, set $q=\alpha z$ in~\eqref{eq:GN}.  The
leading equation is
\[
    z=-\frac{L}{L-(L-1)s},
\]
and~\eqref{eq:muq} gives
\[
    \mu_u(\alpha,s)
    =
    -(L-1)(1-s)+O(\alpha),
    \qquad
    \partial_s\mu_u=L-1+O(\alpha)>0
\]
near $s_0$.  The exact value~\eqref{eq:sNexact} belongs to this branch because
$q_-=O(\alpha)$, and there $\mu_u=-1$.  Hence the localized eigenvalue enters
the unit disk as $s$ increases through $s_{-1}^{(N)}$.  The regular
nonconstant branches are already strictly inside, and the constant branch
remains at $1$.  This proves~\eqref{eq:local-equivalence}.
\end{proof}

\begin{corollary}[Leading-order redistribution threshold]
\label{cor:s0practical}
For fixed $L>2$ and $N\ge4$,
\[
 s_0(L)-s_{-1}^{(N)}(\alpha,L)
 =
 \frac{L(L-2)}{2(L-1)^2}\alpha^2+O(\alpha^3)>0
\]
for sufficiently small $\alpha$, where
\begin{equation}
    s_0(L)=\frac{L-2}{L-1}.
    \label{eq:smodel}
\end{equation}
Hence $s=s_0(L)$ is asymptotically sufficient for the full periodic model and
exceeds the exact local eigenvalue-crossing threshold only by $O(\alpha^2)$.  For fixed
$0<L<2$, Corollary~\ref{cor:threshold} shows that the unstabilized operator
$s=0$ is already spectrally stable for sufficiently small $\alpha$.
\end{corollary}

The exact expression \eqref{eq:sNexact} is useful for verifying the finite-$\alpha$
crossing of the model operator, but it is intentionally not used as the
higher-order prescription.  It depends on the specific first-order,
one-cut-cell periodic operator and on $N$, whereas the higher-order benchmark
changes the reconstruction, time integrator, and number of cut cells.  Using
$s_0(L)$ avoids assigning specific to the model finite-$\alpha$ precision to that
transfer.  For reference, at $N=150$ and $L=4$,
$s_0-s_{-1}^{(150)}$ is approximately $0.0050$ for $\alpha=0.10$ and
$0.0122$ for $\alpha=0.15$.

The two statements meet at the limiting boundary $L=2$, but the theorem does
not assert that $s=0$ is spectrally stable there for every finite $\alpha$.
For $L>2$, $s_0(L)$ increases monotonically from $0$ and approaches $1$ as
$L\to\infty$.  Thus the model eigenvalue-crossing threshold can be substantially below
the full-merging value $s=1$ for moderate $L$, while
$s_0(L)\to1$ as $L\to\infty$.

\begin{remark}[redistribution condition versus power boundedness]
Theorem~\ref{thm:criterion} establishes the local condition $\rho(A_s)\le1$
threshold, meaning $\rho(A_s)\le1$.  The matrices $A_s$ are generally nonnormal, so this condition alone does not
imply power boundedness, i.e., $\sup_{n\ge0}\|A_s^n\|<\infty$ in a fixed
matrix norm, and it does not exclude finite transient amplification.
The theorem is also local in $s$ and asymptotic in $\alpha$ at fixed $L$; it
does not prove that every $s\ge s_{-1}^{(N)}(\alpha,L)$ is spectrally stable
for arbitrary finite $\alpha$, nor does it transfer the same threshold to
other fluxes, reconstructions, nonlinear systems, or multidimensional
redistribution rules.  Section~\ref{sec:numerics} therefore reports
finite-time amplification separately from the spectral crossing.
\end{remark}

\begin{figure}[ht]
\centering
\includegraphics[width=\textwidth]{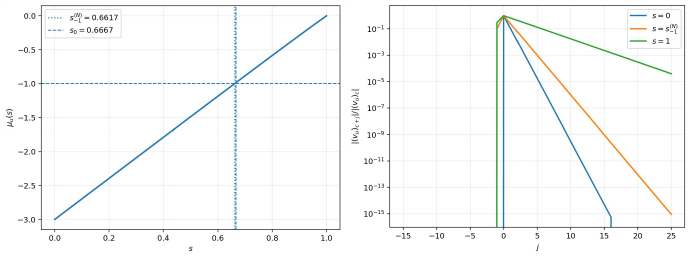}
\caption{Localized spectral branch and eigenvector for $N=150$,
$\alpha=0.10$, and $L=4$.  \textit{Left}: $\mu_u(s)$ with
$\mu_u(s_{-1}^{(150)})=-1$, where
$s_{-1}^{(150)}\approx0.6617$ and $s_0=2/3\approx0.6667$.
\textit{Right}: $|(v_u)_{c+j}|/|(v_u)_c|$ at $s=0$,
$s=s_{-1}^{(150)}$, and $s=1$, where $j$ is a signed periodic offset and
$c+j$ is interpreted modulo $N$.}
\label{fig:spectral_mechanism}
\end{figure}

Figure~\ref{fig:spectral_mechanism} displays the two quantities entering
Theorem~\ref{thm:criterion}: the branch $\mu_u(s)$ and its normalized
eigenvector.  The mode is already localized at $s=0$.  Increasing $s$ moves
$\mu_u(s)$ toward the unit circle, with
$\mu_u(s_{-1}^{(N)})=-1$.  The distinction between
$s_{-1}^{(N)}$ and $s_0$ is the finite-$\alpha$ correction quantified by
\eqref{eq:sNexp}.

\section{Block-Matrix Analysis of Cut-Cell Operators with Diverging Coefficients}
\label{sec:singular}

The preceding calculation has one singular row.  We now isolate the part of
the argument that does not depend on the one-dimensional upwind stencil.
The section has two steps.  First, Theorem~\ref{thm:singular} identifies the
unbounded eigenvalue cluster from the matrix $\Gamma$ in the leading cut-cell block.
Second, Theorem~\ref{thm:proj} shows that the associated invariant subspace
concentrates on the cut-cell coordinates.  The later cut-cell results apply
these two statements with concrete finite-volume blocks.

Let $A_\alpha\in\mathbb C^{N\times N}$ be a family of matrices.  After
reordering the unknowns, suppose the first $m$ coordinates represent geometric
cut-cells and the remaining $N-m$ coordinates are regular.  Write
\begin{equation}
 A_\alpha
 =
 \begin{pmatrix}
     -\alpha^{-p}\Gamma+R_\alpha & B_\alpha\\
     C_\alpha & Q_\alpha
 \end{pmatrix},
 \qquad p>0.
 \label{eq:singular_block}
\end{equation}
The matrix $\Gamma\in\mathbb C^{m\times m}$ is the leading coefficient
of the leading cut-cell block.  The cut-cell coordinate subspace is
\[
    \mathcal C=\mathbb C^m\oplus\{0\}.
\]
Throughout this section, $I_k$ denotes the $k\times k$ identity matrix,
$I$ denotes an identity matrix whose dimension is clear from context, and
$\|\cdot\|$ denotes any fixed subordinate matrix norm.

\begin{assumption}[Asymptotic cut-cell block structure]
\label{ass:sb}
As $\alpha\to0$,
\[
    \|\alpha^pR_\alpha\|\to0,\qquad
    \|B_\alpha\|=O(\alpha^{-p}),\qquad
    \|C_\alpha\|=O(1),\qquad
    \|Q_\alpha\|=O(1),
\]
and $\Gamma$ is invertible.  The dimensions $m$ and $N$ are fixed.
\end{assumption}

The upper-right block may have the same asymptotic order as the leading cut-cell block.
This is needed for conservative cut-cell methods because a cut-cell row with coefficients that diverge as $\alpha\to0$ contains
both a large diagonal coefficient and large coefficients multiplying its
regular neighbors.  The lower blocks remain bounded because regular-cell
updates are not divided by a vanishing volume.

\begin{theorem}[Spectrum generated by the leading cut-cell block]
\label{thm:singular}
Under Assumption~\ref{ass:sb}, $A_\alpha$ has exactly $m$ unbounded
eigenvalues, counted with algebraic multiplicity.  Let
$\gamma_1,\ldots,\gamma_m$ denote the eigenvalues of $\Gamma$, repeated
according to algebraic multiplicity.  The unbounded eigenvalues of $A_\alpha$
can be labeled $\mu_1(\alpha),\ldots,\mu_m(\alpha)$ so that
\begin{equation}
    \alpha^p\mu_k(\alpha)\longrightarrow-\gamma_k,
    \qquad k=1,\ldots,m.
    \label{eq:scaled-spectrum}
\end{equation}
All remaining eigenvalues are bounded.  If, in addition, $\Gamma$ is
diagonalizable and $R_\alpha=O(1)$, then the unbounded eigenvalues may be
matched with the eigenvalues $\gamma_k$ of $\Gamma$ so that
\begin{equation}
    \mu_k(\alpha)
    =
    -\frac{\gamma_k}{\alpha^p}+O(1).
    \label{eq:large-O1}
\end{equation}
\end{theorem}

\begin{proof}
For $|\mu|>\|Q_\alpha\|+1$, Schur complementation gives
\[
 \det(\mu I-A_\alpha)
 =
 \det(\mu I-Q_\alpha)\det \mathcal S_\alpha(\mu),
\]
where
\[
 \mathcal S_\alpha(\mu)
 =
 \mu I_m+\alpha^{-p}\Gamma-R_\alpha
 -
 B_\alpha(\mu I-Q_\alpha)^{-1}C_\alpha.
\]
Set $\nu=\alpha^p\mu$.  Uniformly for $\nu$ in compact subsets of
$\mathbb C\setminus\{0\}$,
\[
 \alpha^p\mathcal S_\alpha(\nu/\alpha^p)
 =
 \nu I_m+\Gamma+o(1).
\]
Indeed, $\alpha^pR_\alpha\to0$, and
\[
 \alpha^p
 B_\alpha(\nu\alpha^{-p}I-Q_\alpha)^{-1}C_\alpha
 =
 O(\alpha^p).
\]
Choose disjoint contours enclosing the distinct eigenvalue clusters of
$-\Gamma$ and excluding $0$.  Uniform convergence of the determinant on each
contour and Rouch\'e's theorem preserve the algebraic number of roots inside
each contour.  Hence there are exactly $m$ such eigenvalues, counted with algebraic
multiplicity, and they can be labeled so that
$\alpha^p\mu_k(\alpha)\to-\gamma_k$.

Suppose a further eigenvalue were unbounded.  Since
$\|A_\alpha\|=O(\alpha^{-p})$, after passing to a subsequence its scaled value
$\nu_\alpha=\alpha^p\mu_\alpha$ has a finite limit.  A nonzero limit would
belong to $-\sigma(\Gamma)$ and is already counted.  If
$\nu_\alpha\to0$ while $|\mu_\alpha|\to\infty$, then
\[
    \alpha^p\mathcal S_\alpha(\mu_\alpha)\longrightarrow\Gamma,
\]
which is invertible, a contradiction.  Thus the remaining spectrum is bounded.

For the final refinement, assume that $\Gamma$ is diagonalizable and
$R_\alpha=O(1)$.  On the unbounded spectral scale, the Schur equation can be
written
\[
 \det\!\left(\mu I_m+\alpha^{-p}\Gamma+H_\alpha(\mu)\right)=0,
 \qquad
 H_\alpha(\mu)
 =-R_\alpha-B_\alpha(\mu I-Q_\alpha)^{-1}C_\alpha,
\]
and $\|H_\alpha(\mu)\|\le C$ uniformly whenever
$|\mu|\asymp\alpha^{-p}$.  Write $\Gamma=V\Lambda V^{-1}$ with $\Lambda$ diagonal and let
$\kappa(V)=\|V\|\,\|V^{-1}\|$.  For
$M_0(\mu)=\mu I_m+\alpha^{-p}\Gamma$,
\[
    \|M_0(\mu)^{-1}\|
    \le
    \frac{\kappa(V)}
    {\operatorname{dist}(\mu,-\alpha^{-p}\sigma(\Gamma))}.
\]
Here $\operatorname{dist}(z,S):=\inf_{w\in S}|z-w|$ denotes the distance
from a complex number $z$ to a set $S\subset\mathbb C$.
Hence, if the distance in the denominator exceeds $\kappa(V)C$, then
$\|M_0(\mu)^{-1}H_\alpha(\mu)\|<1$ and the Neumann-series criterion
shows that
$M_0(\mu)+H_\alpha(\mu)$ is invertible.  Every unbounded eigenvalue must
therefore lie within an $O(1)$ neighborhood of
$-\alpha^{-p}\sigma(\Gamma)$.  Combined with the algebraic multiplicity count
already established for each cluster, this yields~\eqref{eq:large-O1}.
\end{proof}

\begin{theorem}[Convergence of the cut-cell invariant subspace]
\label{thm:proj}
Under Assumption~\ref{ass:sb}, let $\mathcal U_\alpha$ be the invariant subspace
associated with the $m$ unbounded eigenvalues.  Then
\begin{equation}
    \|P_{\mathcal U_\alpha}-P_{\mathcal C}\|=O(\alpha^p),
    \label{eq:proj}
\end{equation}
where $P_{\mathcal V}$ denotes the Euclidean orthogonal projector onto a
subspace $\mathcal V$.
\end{theorem}

\begin{proof}
Scale the matrix:
\[
 \widetilde A_\alpha:=\alpha^pA_\alpha
 =
 \begin{pmatrix}
     -\Gamma+\alpha^pR_\alpha & \alpha^pB_\alpha\\
     \alpha^pC_\alpha & \alpha^pQ_\alpha
 \end{pmatrix}.
\]
By Theorem~\ref{thm:singular}, the nonzero spectral cluster converges to
$-\sigma(\Gamma)$ and remains separated from the cluster at $0$.  Choose a
fixed positively oriented contour $\mathfrak C$ enclosing
$-\sigma(\Gamma)$ and excluding $0$.  The Riesz projector
\[
    \Pi_\alpha
    =
    \frac{1}{2\pi i}
    \int_{\mathfrak C}
    (zI-\widetilde A_\alpha)^{-1}\,dz
\]
has range $\mathcal U_\alpha$ and rank $m$.

Write
\[
 zI-\widetilde A_\alpha
 =
 \begin{pmatrix}
 X_\alpha(z) & -\alpha^pB_\alpha\\
 -\alpha^pC_\alpha & Y_\alpha(z)
 \end{pmatrix},
\]
where
\[
 X_\alpha(z)=zI+\Gamma-\alpha^pR_\alpha,
 \qquad
 Y_\alpha(z)=zI-\alpha^pQ_\alpha.
\]
Since $\mathfrak C$ stays a positive distance from $0$,
$Y_\alpha(z)^{-1}=O(1)$ uniformly on $\mathfrak C$.  The Schur complement
\[
 \mathcal S_\alpha(z)
 =
 X_\alpha(z)
 -
 \alpha^{2p}B_\alpha Y_\alpha(z)^{-1}C_\alpha
\]
satisfies
\[
 \mathcal S_\alpha(z)
 =
 zI+\Gamma+o(1)
\]
uniformly on $\mathfrak C$, because
$\alpha^{2p}B_\alpha Y_\alpha(z)^{-1}C_\alpha=O(\alpha^p)$.
Hence $\mathcal S_\alpha(z)^{-1}=O(1)$ uniformly on the contour.

The block inverse formula now gives
\[
 (zI-\widetilde A_\alpha)^{-1}_{21}
 =
 Y_\alpha(z)^{-1}\alpha^pC_\alpha\mathcal S_\alpha(z)^{-1}
 =
 O(\alpha^p),
\]
and
\[
 (zI-\widetilde A_\alpha)^{-1}_{22}
 =
 Y_\alpha(z)^{-1}
 +
 Y_\alpha(z)^{-1}\alpha^pC_\alpha\mathcal S_\alpha(z)^{-1}
 \alpha^pB_\alpha Y_\alpha(z)^{-1}.
\]
The second term is $O(\alpha^p)$.  The first term,
$Y_\alpha(z)^{-1}$, is analytic inside $\mathfrak C$ because the spectrum of
$\alpha^pQ_\alpha$ remains near $0$, outside the contour.  Its contour
integral is therefore zero.  It follows that
\[
    (\Pi_\alpha)_{21}=O(\alpha^p),
    \qquad
    (\Pi_\alpha)_{22}=O(\alpha^p).
\]
Similarly,
\[
    (\Pi_\alpha)_{11}=I_m+o(1),
\]
because the contour encloses all $m$ eigenvalues of $-\Gamma$.

For small $\alpha$, $(\Pi_\alpha)_{11}$ is invertible.  The $m$ vectors
$\Pi_\alpha(x,0)^\top$ therefore span $\operatorname{range}(\Pi_\alpha)$ and
show that
\[
    \mathcal U_\alpha
    =
    \{(x,K_\alpha x):x\in\mathbb C^m\},
    \qquad
    K_\alpha
    =
    (\Pi_\alpha)_{21}(\Pi_\alpha)_{11}^{-1}
    =
    O(\alpha^p).
\]
The standard orthogonal-projector formula for the graph of $K_\alpha$ then
gives
$\|P_{\mathcal U_\alpha}-P_{\mathcal C}\|=O(\|K_\alpha\|)=O(\alpha^p)$.
\end{proof}

\begin{remark}
The projector theorem concerns the whole invariant subspace and therefore does
not require the cut-cell eigenvalues to be simple.  This distinction matters
when two cut cells have the same leading cut-cell coefficient: individual
eigenvectors may not have a unique limiting labeling, while the $m$-dimensional unbounded spectral subspace still converges to the
subspace spanned by the cut-cell coordinates.
\end{remark}

Theorems~\ref{thm:singular} and~\ref{thm:proj} give the following general result:
a singular leading cut-cell block creates a finite cluster of unbounded eigenvalues, and
the corresponding invariant subspace concentrates on the coordinates where
the singular scaling occurs.  Computationally, this identifies a
low-dimensional object to analyze: when the number of cut cells
$m\ll N$, the singular spectral behavior is carried by an $m$-dimensional
invariant subspace associated with the cut-cell eigenvalues rather than requiring analysis of the full
$N$-dimensional state space.  The next sections recover the concrete
cut-cell consequences.

\section{Multiple Cut Cells}
\label{sec:multi}

Let
\[
    J=\{j_1,\ldots,j_m\}
\]
be non-adjacent cut cells.  Assume their volume fractions shrink at comparable rates.  Write the volume
fraction of cut cell $j_k$ as $\alpha_k=\theta_k\alpha$, where $\alpha$ is a
common small parameter, $\theta_k$ is a fixed positive scale factor, and the
scale factors satisfy fixed bounds $0<\theta_-\le\theta_k\le\theta_+<\infty$:
\begin{equation}
    \alpha_k=\theta_k\alpha,
    \qquad
    0<\theta_-\le\theta_k\le\theta_+<\infty,
    \label{eq:comparable-alpha}
\end{equation}
with fixed $m$, $N$, and background CFL $\lambda\in(0,1)$.  Reorder the
unknowns so the cut-cell coordinates come first.

For first-order positive-speed upwind, non-adjacency means that the upwind
neighbor of a cut cell is regular.  The cut-cut block is therefore diagonal:
\[
    (A_\alpha)_{JJ}
    =
    I_m-\frac1\alpha\Gamma,
    \qquad
    \Gamma=
    \operatorname{diag}\!\left(
        \frac{\lambda}{\theta_1},\ldots,
        \frac{\lambda}{\theta_m}
    \right).
\]
The cut-to-regular block is $O(\alpha^{-1})$, while all regular rows are
bounded.  Thus Assumption~\ref{ass:sb} holds with $p=1$.

\begin{corollary}[Invariant subspace for separated cut cells]
\label{thm:multi}
Under~\eqref{eq:comparable-alpha}, the periodic first-order upwind operator has
exactly $m$ unbounded eigenvalues.  After matching with multiplicity,
\begin{equation}
    \mu_k(\alpha)
    =
    -\frac{\lambda}{\theta_k\alpha}+O(1),
    \qquad k=1,\ldots,m.
    \label{eq:multi-eigs}
\end{equation}
If $\mathcal U_\alpha$ is the invariant subspace associated with these eigenvalues
and
\[
    \mathcal C=\operatorname{span}\{e_{j_1},\ldots,e_{j_m}\},
\]
then
\begin{equation}
    \|P_{\mathcal U_\alpha}-P_\mathcal C\|=O(\alpha).
    \label{eq:multi-proj}
\end{equation}
\end{corollary}

\begin{proof}
The matrix $\Gamma$ is diagonal and hence diagonalizable, while the remainder in
the cut-cut block is $I_m=O(1)$.  Apply
Theorems~\ref{thm:singular} and~\ref{thm:proj}.
\end{proof}

Equation~\eqref{eq:multi-proj} is the invariant-subspace version of the
one-cell localization theorem.  When two cut cells have the same leading
coefficient, individual eigenvectors in the unbounded spectral cluster need
not select unique coordinate vectors.  The invariant statement is that the
entire unbounded spectral subspace converges to the span of the cut-cell
coordinates; this subspace remains well defined when leading cut-cell
eigenvalues are repeated.

Adjacent cut cells can introduce
diverging couplings inside the leading cut-cell block and therefore require a different
$\Gamma$.  Allowing $m$ or $N$ to grow with $\alpha^{-1}$ is also outside the
fixed-dimensional theorem.  Likewise, widely separated volume scales require
either several small parameters or a rescaling different from
\eqref{eq:comparable-alpha}.

\section{One Cut Cell in \texorpdfstring{$d$}{d} Dimensions}
\label{sec:ddim}

Consider a periodic first-order upwind discretization in $d$ dimensions.  Let
$\lambda_r$ be the directional background CFL number in coordinate direction
$r$, and define the summed background CFL number
\[
    \lambda_\Sigma:=\sum_{r=1}^d\lambda_r.
\]
The corresponding local cut-cell CFL number is
$\lambda_{\alpha,\Sigma}:=\lambda_\Sigma/\alpha$.
For one cut cell $c$ of volume fraction $\alpha$, the cut-cell row is
\begin{equation}
    (A_\alpha U)_c
    =
    \left(1-\frac{\lambda_\Sigma}{\alpha}\right)U_c
    +
    \sum_{r=1}^d
    \frac{\lambda_r}{\alpha}U_{c-\hat r},
    \label{eq:Add}
\end{equation}
where $\hat r$ denotes one grid-index step in coordinate direction $r$, so
$c-\hat r$ is the upwind neighbor of the cut cell in that direction.  Regular
rows remain $O(1)$.

After moving the cut-cell coordinate first,
Assumption~\ref{ass:sb} holds with $m=1$, $p=1$, scalar $\Gamma=\lambda_\Sigma$, and bounded remainder in the cut-cell diagonal.

\begin{corollary}[Multidimensional localized cut-cell mode]
\label{thm:ddim}
For fixed grid size and fixed directional CFL numbers, the operator has one
unbounded eigenvalue
\begin{equation}
    \mu_\alpha
    =-\lambda_{\alpha,\Sigma}+O(1)
    =-\frac{\lambda_\Sigma}{\alpha}+O(1),
    \label{eq:mudd}
\end{equation}
while all remaining eigenvalues are bounded.  If the associated eigenvector is
normalized by $(v_\alpha)_c=1$, then
\begin{equation}
    \|v_\alpha-e_c\|_2=O(\alpha).
    \label{eq:vdd}
\end{equation}
\end{corollary}

\begin{proof}
Apply Theorems~\ref{thm:singular} and~\ref{thm:proj} with
$\Gamma=\lambda_\Sigma$.  Since the scalar matrix $\Gamma$ is diagonalizable and the cut-cell
remainder is bounded, the $O(1)$ eigenvalue refinement applies.
\end{proof}

The multidimensional corollary is a leading-order spectral result.  The geometry of a real embedded boundary affects face areas,
neighboring cut cells, reconstruction, and the redistribution neighborhoods.
Those features enter through the actual assembled matrix $\Gamma$ and cannot be
replaced by the scalar $L$ without a separate derivation.

\section{Finite-Volume Conditions for the Asymptotic Cut-Cell Block}
\label{sec:characterization}

The abstract block assumptions are useful only if they can be checked from a
discretization.  We therefore state sufficient conditions at the assembled
operator level and then derive a local formula for a one-dimensional
bounded-flux scheme.  We state sufficient conditions rather than an equivalence because cancellations
among face contributions can invalidate conclusions based only on face-level
scaling estimates.

Let $J_c$ and $J_r$ denote the cut-cell and regular-cell index sets,
respectively.  For index sets $X$ and $Y$, $(A_\alpha)_{XY}$ denotes the block
of $A_\alpha$ with rows in $X$ and columns in $Y$.  Assume $\alpha_k=\theta_k\alpha$ with the fixed
bounds in \eqref{eq:comparable-alpha}.

\begin{theorem}[Sufficient assembled-operator conditions]
\label{thm:characterization}
Suppose that, after ordering the cut-cell coordinates first, there exists
$p>0$ and an invertible matrix $\Gamma$ such that
\begin{equation}
    \alpha^p(A_\alpha)_{J_cJ_c}\longrightarrow-\Gamma,
    \label{eq:JJlimit}
\end{equation}
and
\begin{equation}
    \|(A_\alpha)_{J_cJ_r}\|=O(\alpha^{-p}),
    \qquad
    \|(A_\alpha)_{J_rJ_c}\|+\|(A_\alpha)_{J_rJ_r}\|=O(1).
    \label{eq:blockbounds}
\end{equation}
Then $A_\alpha$ satisfies Assumption~\ref{ass:sb}.
\end{theorem}

\begin{proof}
Set
\[
 R_\alpha=(A_\alpha)_{J_cJ_c}+\alpha^{-p}\Gamma,\quad
 B_\alpha=(A_\alpha)_{J_cJ_r},\quad
 C_\alpha=(A_\alpha)_{J_rJ_c},\quad
 Q_\alpha=(A_\alpha)_{J_rJ_r}.
\]
Equation~\eqref{eq:JJlimit} is exactly
$\alpha^pR_\alpha\to0$, and~\eqref{eq:blockbounds} supplies the remaining
bounds.
\end{proof}

Let $C_c$ denote the cut-cell control volume, $|C_c|$ its volume,
$\partial C_c$ its boundary, $f$ a face on that boundary, $|f|$ the face
measure, and $F_f(U)$ the numerical flux through face $f$ with the chosen
outward orientation.  A conservative finite-volume update on $C_c$ has the form
\begin{equation}
    U_c^{n+1}
    =
    U_c^n
    -
    \frac{\Delta t}{|C_c|}
    \sum_{f\subset\partial C_c}|f|\,F_f(U).
    \label{eq:fv_general}
\end{equation}
The singular scaling can therefore come from the small volume
$|C_c|$, from reconstruction coefficients inside $F_f$, from shrinking or
nonshrinking face measures, or from combinations of these effects.  The
assembled conditions above keep all of those factors visible instead of
assuming that the volume fraction alone determines the power of $\alpha$.

\section{One-Dimensional Bounded-Flux Schemes}
\label{sec:verify}

We now connect the block analysis to the local flux coefficients in the
simplest setting where an explicit formula is available.  Consider a scalar
one-dimensional conservation law with one cut cell of width $\alpha h$.
Linearize the two numerical fluxes adjacent to the cut cell.  Let
$\beta_{c\pm1/2,j}(\alpha)$ denote the resulting coefficient multiplying cell
value $U_j$, where $j$ ranges over the finite flux stencil:
\begin{equation}
    F_{c+1/2}(U)
    =
    \sum_j \beta_{c+1/2,j}(\alpha)U_j,
    \qquad
    F_{c-1/2}(U)
    =
    \sum_j \beta_{c-1/2,j}(\alpha)U_j.
    \label{eq:linearized_flux}
\end{equation}
Only finitely many coefficients are nonzero.

\begin{theorem}[Leading cut-cell coefficient for bounded-flux schemes]
\label{thm:gamma_formula}
Assume
\[
    \beta_{c\pm1/2,j}(\alpha)=O(1)
\]
uniformly as $\alpha\to0$, all regular rows of the assembled update are
$O(1)$, and
\begin{equation}
    g_\alpha
    :=
    \frac{\Delta t}{h}
    \left(
      \beta_{c+1/2,c}(\alpha)-\beta_{c-1/2,c}(\alpha)
    \right)
    =
    \Gamma_0+O(\alpha),
    \qquad
    \Gamma_0\ne0.
    \label{eq:g-alpha}
\end{equation}
Then the assembled operator satisfies Assumption~\ref{ass:sb} with
$m=1$, $p=1$, scalar leading cut-cell coefficient $\Gamma_0$, bounded
$R_\alpha$, $B_\alpha=O(\alpha^{-1})$, and bounded lower blocks.
\end{theorem}

\begin{proof}
The cut-cell update is
\[
 U_c^{n+1}
 =
 U_c^n
 -
 \frac{\Delta t}{\alpha h}
 \bigl(F_{c+1/2}(U)-F_{c-1/2}(U)\bigr).
\]
The coefficient multiplying $U_c$ is therefore
\[
    1-\frac{g_\alpha}{\alpha}
    =
    -\frac{\Gamma_0}{\alpha}+O(1).
\]
Every other coefficient in the cut-cell row is $O(\alpha^{-1})$ because the
linearized flux coefficients are bounded.  By assumption, the regular rows
remain bounded.  These are precisely the block estimates of
Assumption~\ref{ass:sb}.
\end{proof}

By~\eqref{eq:g-alpha}, $\Gamma_0$ is the leading coefficient of $U_c$ in
the dimensionless flux difference
$(\Delta t/h)(F_{c+1/2}-F_{c-1/2})$ as $\alpha\to0$.
For positive-speed first-order upwind,
\[
    \beta_{c+1/2,c}=a,\qquad \beta_{c-1/2,c}=0,
\]
so $\Gamma_0=\lambda$.  The abstract theorem then recovers the leading
unbounded eigenvalue $-\lambda/\alpha+O(1)$, consistent with
Theorem~\ref{thm:eigenvalue}.

The value $2$ and the spectral-threshold formula in
Section~\ref{sec:criterion} are specific to the complete characteristic
equation of the first-order upwind model with cell merging.  The leading cut-cell coefficient
$\Gamma_0$ alone determines the leading singular spectral scale, not a
scheme-independent redistribution threshold.

If reconstruction introduces another inverse power of the cut-cell scale,
then the cut-cell row can be more singular.  For example, a reconstruction
coefficient of size $O(\alpha^{-1})$ combined with division by the cut-cell
volume produces $O(\alpha^{-2})$ entries.  In that case the same block analysis is
applied with $p=2$ after computing the nonzero scaled limit
$\alpha^2(A_\alpha)_{J_cJ_c}$.  This observation concerns the spectral scaling
of the assembled operator; it does not by itself imply a redistribution threshold
for a higher-order reconstruction.

\section{MUSCL--Minmod Extension}
\label{sec:muscl-theory}

The first-order calculation isolates the localized eigenvalue behavior, but the
higher-order benchmark in Section~\ref{sec:numerics} uses piecewise-linear
MUSCL reconstruction with the minmod limiter and SSPRK2.  We now analyze the
localized cut-cell mode of that stage map directly.

Consider the one-cut periodic mesh of Section~\ref{sec:1d} with cut cell
$c=0$, upstream regular cell $l=N-1$, and fixed local CFL $L>2$, so the
background CFL is $\lambda=\alpha L$.  Let $\mathcal E_s$ denote one
forward Euler MUSCL--minmod stage followed by the scalar blend toward the merged state
\[
    \mathcal E_s(V)
    =(1-s)\widehat U(V)+s\mathcal R(\widehat U(V)).
\]
Because minmod is homogeneous for every real scalar,
$\operatorname{minmod}(\gamma r_1,\gamma r_2)
=\gamma\operatorname{minmod}(r_1,r_2)$, the map $\mathcal E_s$ is
homogeneous of degree one.  We therefore study nonzero vectors $v$ satisfying
the nonlinear eigenvalue equation
\[
    \mathcal E_s(v)=\nu v.
\]

At the unstable crossing $\nu=-1$, normalize $v_0=1$ and write
\[
    v_{N-1}=\alpha a.
\]
For the localized branch, the minmod slope vanishes at the cut cell and at
the downstream regular cells, while the upstream cell $N-1$ selects its
backward slope.  The next theorem analyzes this particular branch of the
piecewise-linear minmod map and verifies that these limiter choices remain
consistent for sufficiently small $\alpha$.

\begin{theorem}[MUSCL--minmod crossing on a fixed minmod branch]
\label{thm:muscl-crossing}
Fix $L>2$.  Use the minmod branch described above and impose zero incoming
reconstructed perturbation at the remote face entering the upstream cell
$N-1$.  Then the blended MUSCL--minmod Euler-stage map has a nonzero real
vector $v$ satisfying $\mathcal E_s(v)=-v$ at the unique value
\begin{equation}
 s_{\mathrm M}^{\mathrm{loc}}(\alpha,L)
 =
 \frac{(L-2)(1+\alpha)(3L\alpha-4)}
 {3L^2\alpha^2+3L^2\alpha
 -6L\alpha^2-8L\alpha-4L+4\alpha+4}.
 \label{eq:sM-local}
\end{equation}
The upstream component is exactly
\begin{equation}
    v_{N-1}=\frac{L-2}{2}\alpha
    \label{eq:muscl-upstream-component}
\end{equation}
in this localized cut-cell model.  Moreover,
\begin{equation}
 s_{\mathrm M}^{\mathrm{loc}}(\alpha,L)
 =
 \frac{L-2}{L-1}
 -
 \frac{L(L-2)}{4(L-1)^2}\alpha
 +O(\alpha^2),
 \label{eq:sM-expansion}
\end{equation}
and the crossing is transverse:
$\partial_s\nu=L-1+O(\alpha)>0$.
\end{theorem}

\begin{proof}
Normalize $v_0=1$ and write $v_{N-1}=\alpha a$.  At the cut cell,
$v_0-v_{N-1}>0$ while the downstream difference has the opposite sign, so
the minmod slope is zero.  At the upstream cell $N-1$, the backward
difference quotient is $O(\alpha)$ and the forward quotient toward the cut
cell is $O(1)$.  For $a$ near $(L-2)/2>0$, both are positive and minmod
therefore selects the backward slope.

Under the local closure, the provisional values entering cell merging are
\begin{align}
    \widehat v_{N-1}
    &=
    \alpha a-\frac32\alpha^2La,
    \label{eq:muscl-vhat-left}\\
    \widehat v_0
    &=
    1-L+\frac32\alpha La.
    \label{eq:muscl-vhat-cut}
\end{align}
Let
\[
    Q=\frac{\widehat v_{N-1}+\alpha\widehat v_0}{1+\alpha}.
\]
The crossing equations
\begin{equation}
    -\alpha a=(1-s)\widehat v_{N-1}+sQ,
    \qquad
    -1=(1-s)\widehat v_0+sQ
    \label{eq:muscl-crossing-system}
\end{equation}
are two algebraic equations in $(a,s)$.  Eliminating $a$ gives
\[
    a=\frac{L-2}{2}
\]
and the unique nearby root~\eqref{eq:sM-local}.  Expanding at fixed $L$
gives~\eqref{eq:sM-expansion}.

The strict inequalities defining these limiter choices persist for sufficiently
small $\alpha$.  On this branch the stage map is linear and the localized
eigenvalue satisfies
\[
    \nu(s)=-(L-1)(1-s)+O(\alpha).
\]
Hence
\[
    \partial_s\nu=L-1+O(\alpha)>0.
\]
\end{proof}

\begin{remark}[Periodic closure]
\label{rem:muscl-periodic-closure}
For the full periodic one-cut mesh, the localized downstream tail has
$|v_i|=O(\alpha^i)$ away from the cut cell until the remote closure region.
Thus the reconstructed perturbation returning to the upstream cell is
higher order in $\alpha$ for fixed $N$.  Because minmod is piecewise linear
and its active branch can change in that remote closure region, we do not use
a smooth implicit-function argument there and do not claim a theorem for the
exact finite-$N$ periodic crossing.  Instead, Section~\ref{sec:numerics}
checks the full periodic nonlinear eigenvalue equations directly.  At $N=150$
the measured crossing agrees with~\eqref{eq:sM-local} to within
$1.2\times10^{-15}$ for both finite-$\alpha$ cases used in the benchmark.
\end{remark}

\begin{corollary}[SSPRK2 crossing]
\label{cor:muscl-ssprk2}
For the minmod branch used in Theorem~\ref{thm:muscl-crossing}, let
\[
    \mathcal S_s(V)
    =
    \frac12V+\frac12\mathcal E_s(\mathcal E_s(V))
\]
be the SSPRK2 map.  If a nonzero real vector $v$ satisfies $\mathcal E_s(v)=\nu v$,
then
\begin{equation}
    \mathcal S_s(v)
    =
    \frac{1+\nu^2}{2}\,v.
    \label{eq:ssprk2-eigenvalue}
\end{equation}
Hence the unstable real localized branch reaches unit amplification at the
same crossing $\nu=-1$, so on this minmod branch its redistribution
threshold is $s_{\mathrm M}^{\mathrm{loc}}(\alpha,L)$.
\end{corollary}

\begin{proof}
Homogeneity gives
$\mathcal E_s(\nu v)=\nu\mathcal E_s(v)=\nu^2v$.  Substitution into the
SSPRK2 formula yields~\eqref{eq:ssprk2-eigenvalue}.  For real $\nu<-1$,
$(1+\nu^2)/2>1$, with equality at $\nu=-1$.
\end{proof}

\begin{remark}[Scope of the MUSCL result]
\label{rem:muscl-scope}
Theorem~\ref{thm:muscl-crossing} concerns one specified branch of the
piecewise-linear minmod map with cell merging.  It does
not assert global nonlinear stability or power boundedness of the full
MUSCL--SSPRK2 method.  Its role is to connect the higher-order benchmark to
the theory: the same leading value $s_0(L)$ appears, and the MUSCL
calculation supplies the first finite-$\alpha$ correction.
\end{remark}

\section{Numerical Experiments}
\label{sec:numerics}

The numerical study has four components.  We first evaluate the spectrum of
the exact periodic upwind operator with cell merging and compute
its finite-time induced amplification near the eigenvalue-crossing threshold.  We then
assemble a two-dimensional rotated-channel cut-cell operator and determine
its redistribution crossing directly from the full spectrum.  Next we test
the MUSCL--minmod threshold from Theorem~\ref{thm:muscl-crossing} in the same
second-order discretization used to derive it, measuring convergence and
fixed-error cell count against weighted-SRD comparators.  Finally, we retain
the broader layout and 100-period tests of the leading prescription $s_0$ and
add a 100-period check of the finite-$\alpha$ MUSCL value $s_{\mathrm M}$.

\subsection{Full-matrix spectral crossing}

Figure~\ref{fig:spectral_mechanism}, placed after
Theorem~\ref{thm:criterion}, shows the localized eigenvalue branch and the rest of the spectrum for
$N=150$, $\alpha=0.10$, and $L=4$.  The localized eigenvalue begins near
$1-L=-3$, crosses $\mu=-1$ at
\[
    s_{-1}^{(150)}\approx0.661654,
\]
and then lies inside the unit disk.  The leading prescription
$s_0=(L-2)/(L-1)=2/3$ is slightly larger, consistent with the negative
$O(\alpha^2)$ correction in~\eqref{eq:sNexp}.  The associated eigenvectors
remain localized near the cut cell throughout the crossing.

Figure~\ref{fig:spectrum-unit-disk} plots the spectrum in the complex
$\mu$-plane.  The orange curve traces the localized branch $\mu_u(s)$.
The gray markers are the remaining individual eigenvalues of the full matrix
at $s=s_{-1}^{(N)}$; together they form the inner spectral loop.  At $s=0$,
$\mu_u(s)$ lies outside $|\mu|=1$; at $s=s_{-1}^{(N)}$, $\mu_u=-1$; and
for larger $s$ the branch lies inside the unit disk.  The black point at
$\mu=1$ is the conservation eigenvalue.

\begin{figure}[ht]
\centering
\includegraphics[width=0.82\textwidth]{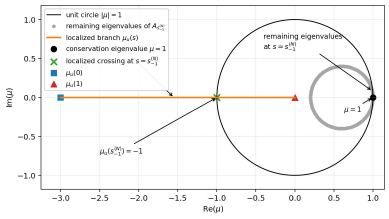}
\caption{Localized spectral crossing for $N=150$, $\alpha=0.10$, and $L=4$.
The unit circle $|\mu|=1$ is the spectral-stability boundary.  The orange
curve traces the localized branch $\mu_u(s)$ for $0\le s\le1$.  The gray
markers denote the remaining individual eigenvalues of the full matrix
$A_{s_{-1}^{(N)}}$ evaluated at the exact crossing
$s_{-1}^{(150)}=0.661654\ldots$; collectively they form the inner spectral
loop and remain strictly inside the unit circle.  The black point at $\mu=1$ is the
conservation eigenvalue, and the green marker identifies the localized
crossing $\mu_u(s_{-1}^{(150)})=-1$.  The asymptotic threshold
$s_0=2/3$ lies slightly above this finite-$\alpha$ crossing.}
\label{fig:spectrum-unit-disk}
\end{figure}

To visualize the asymptotic correction itself, Figure~\ref{fig:threshold_asymptotics}
plots the exact crossing~\eqref{eq:sNexact} of the full periodic operator for
$L\in\{3,4,6\}$ and $N=150$.  The left panel shows convergence toward the
leading value $s_0(L)$.  The right panel scales the difference by
$\alpha^2$; as $\alpha\to0$ the curves approach
\[
    \frac{L(L-2)}{2(L-1)^2},
\]
the coefficient predicted by~\eqref{eq:sNexp}.  This figure is a
visualization of the exact formula, not independent evidence for it.

\begin{figure}[ht]
\centering
\includegraphics[width=\textwidth]{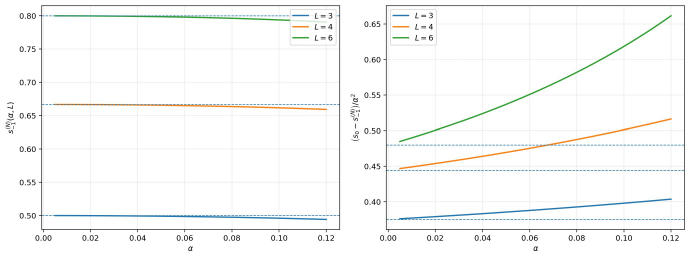}
\caption{Finite-$\alpha$ correction to the local eigenvalue-crossing threshold for
$N=150$.
\textit{Left}: exact crossing $s_{-1}^{(N)}(\alpha,L)$ from
\eqref{eq:sNexact}; dashed horizontal lines are
$s_0(L)=(L-2)/(L-1)$.
\textit{Right}: the scaled difference
$(s_0-s_{-1}^{(N)})/\alpha^2$ approaches
$L(L-2)/(2(L-1)^2)$ as $\alpha\to0$, consistent with
\eqref{eq:sNexp}.}
\label{fig:threshold_asymptotics}
\end{figure}

\subsection{Finite-time amplification near the eigenvalue-crossing threshold}

Because $A_s$ is nonnormal, transient growth can persist after the spectral
crossing.  With the cut cell ordered first, set
\[
    W_\alpha:=\operatorname{diag}(\alpha,1,\ldots,1).
\]
For a matrix $M$, let
\begin{equation}
    \|M\|_\alpha
    :=\sup_{x\ne0}\frac{\|Mx\|_\alpha}{\|x\|_\alpha}
    =\|W_\alpha^{1/2}MW_\alpha^{-1/2}\|_2.
    \label{eq:weighted-induced-norm}
\end{equation}
Figure~\ref{fig:transient-growth} plots $\|A_s^n\|_\alpha$ through
$n=200$ for $N=80$ and fixed $L=4$.

For $\alpha=0.10$, taking $s$ only $0.01$ below the exact crossing gives
$\rho(A_s)=1.0301$ and $\|A_s^{200}\|_\alpha\approx582$; the
eigenvalue outside $\overline{\mathbb D}$ produces geometric modal growth.  At the exact crossing the maximum over the first
$200$ steps is approximately $2.01$.  At the slightly larger model
prescription $s_0=2/3$, the maximum is approximately $2.01$ and
$\|A_s^{200}\|_\alpha$ has returned to $1$.  For $\alpha=0.15$, the
corresponding maxima at the exact crossing and at $s_0$ are approximately
$2.34$ and $2.26$.  These finite-time computations show that crossing the unit circle eliminates
exponential growth of the localized eigenmode while finite nonnormal
amplification remains possible; power boundedness is a separate question.

\begin{figure}[ht]
\centering
\includegraphics[width=0.48\textwidth]{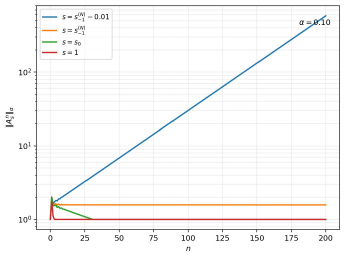}
\hfill
\includegraphics[width=0.48\textwidth]{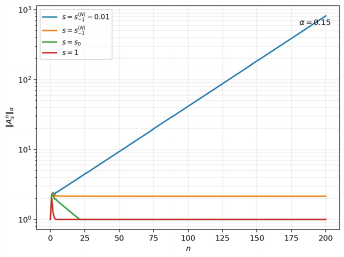}
\caption{Finite-time induced amplification near the spectral crossing for
$N=80$ and $L=4$.  \textit{Left}: $\alpha=0.10$.
\textit{Right}: $\alpha=0.15$.  For $s=s_{-1}^{(N)}-0.01$, $\rho(A_s)>1$ and
$\|A_s^n\|_\alpha$ grows throughout the displayed interval.  For
$s\ge s_{-1}^{(N)}$, the displayed curves remain below $2.35$ for
$0\le n\le200$, showing finite transient amplification despite
$\rho(A_s)\le1$.}
\label{fig:transient-growth}
\end{figure}

\subsection{Two-dimensional rotated-channel redistribution threshold}
\label{sec:2d-experiment}

The preceding experiments use one-dimensional meshes.  To test the
localization result on an actual embedded-boundary geometry, we also
assemble a two-dimensional first-order upwind operator on a rotated channel.
The purpose of this experiment is spectral rather than asymptotic: we ask
whether the assembled two-dimensional cut-cell operator still admits a
blending parameter substantially below full cell merging, and how well
the scalar one-dimensional prediction approximates that fraction.

Let
\[
 \Omega
 =
 \left\{(x,y)\in[0,5/7]\times[0,1]:
 c_-\le y-x\le c_+\right\},
 \qquad c_+=0.78571.
\]
The channel boundaries are parallel to the velocity
\[
    \mathbf a=(-1,-1),
\]
so their physical normal flux is zero.  The geometry is a truncated version
of the $45^\circ$ channel used by Berger~\cite{Berger2015}; truncation removes
the single lower-boundary cut cell that would otherwise terminate at the
outer top corner, while leaving the repeated interior cut-cell geometry
unchanged.  We use a Cartesian background grid with $N=42$ cells per unit
length, so $h=1/42$, and choose $c_-$ separately for each test so the repeated
lower-boundary triangles have prescribed area fraction
\[
    \alpha\in\{0.003698,\,0.02,\,0.05\}.
\]
There are 669 active control volumes and 30 disjoint small-cell neighborhoods.
Each lower-boundary triangle is paired with its Cartesian neighbor immediately
above it.  Figure~\ref{fig:2d-geometry} shows the $\alpha=0.02$ geometry.

\begin{figure}[ht]
\centering
\includegraphics[width=0.68\textwidth]{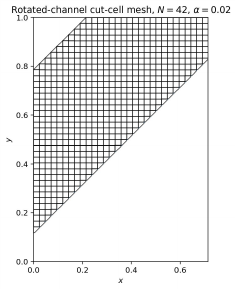}
\caption{Rotated-channel embedded-boundary mesh for $N=42$ and
$\alpha=0.02$.  Hatched cells are the repeated lower-boundary cut cells used
as cut cells in the pairwise redistribution experiment.  Each is paired
with the Cartesian cell immediately above it.}
\label{fig:2d-geometry}
\end{figure}

For first-order upwinding and homogeneous inflow data, let $A$ denote the
assembled two-dimensional Euler update matrix with
\[
    \frac{\Delta t}{h}=0.4,
    \qquad
    \frac{|a_x|\Delta t}{h}
    +
    \frac{|a_y|\Delta t}{h}
    =0.8.
\]
Let $\mathcal R$ denote simultaneous full volume-weighted pair merging on the
30 disjoint neighborhoods.  The blended matrix is
\begin{equation}
    A_s=\big[(1-s)I+s\mathcal R\big]A,
    \qquad 0\le s\le1.
    \label{eq:2d-blend}
\end{equation}
For each repeated triangular cut cell $C_c$, define
\begin{equation}
    L_{\mathrm{def}}
    :=
    \frac{\Delta t}{|C_c|}
    \sum_{f\subset\partial C_c,\;\mathbf a\cdot n_f>0}
    (\mathbf a\cdot n_f)|f|.
    \label{eq:2d-Ldef}
\end{equation}
Thus the diagonal coefficient of the unstabilized cut-cell row is
$1-L_{\mathrm{def}}$.  Substituting this value into the one-dimensional
crossing formula gives
\[
    s_{\mathrm{1D}}
    =
    \frac{L_{\mathrm{def}}-2}{L_{\mathrm{def}}-1}.
\]
For comparison, define the threshold of the assembled two-dimensional matrix
\begin{equation}
    s_{\mathrm{2D}}
    :=
    \inf\{s\in[0,1]:\rho(A_s)\le1\}.
    \label{eq:s2d}
\end{equation}
For $N=42$, $s_{\mathrm{2D}}$ is computed by bisection in $s$ using the
complete eigenvalue spectrum of the $669\times669$ assembled matrix.

\begin{table}[ht]
\centering
\small
\begin{tabular}{rrrrrrrr}
$\alpha$ & $L_{\mathrm{def}}$ & $s_{\mathrm{1D}}$ &
$s_{\mathrm{2D}}$ &
$\rho(A)$ & $\rho(A_{s_{\mathrm{1D}}})$ &
$\rho(A_{s_{\mathrm{2D}}})$ & $\rho(A_1)$\\
\hline
0.003698 & 9.3023 & 0.87955 & 0.87947 & 8.3023 & 0.99932 & 0.99995 & 0.55546\\
0.020000 & 4.0000 & 0.66667 & 0.66796 & 3.0000 & 1.00484 & 0.99872 & 0.56297\\
0.050000 & 2.5298 & 0.34633 & 0.36272 & 1.5298 & 1.02387 & 0.99829 & 0.55794\\
\end{tabular}
\caption{redistribution threshold on the two-dimensional rotated-channel
operator.  The one-dimensional scalar prediction is nearly exact for the
smallest triangles, but slightly underestimates the assembled
two-dimensional threshold for the two larger cut-cell fractions.  Full pair
merging places the spectrum substantially farther inside the unit disk than
is required by the spectral criterion.}
\label{tab:2d-spectral}
\end{table}

Figure~\ref{fig:2d-spectral} shows the full dependence of the spectral radius
on the blending parameter.  For the most restrictive geometry,
$\alpha=0.003698$, the scalar prediction and the assembled two-dimensional
threshold differ by less than $8\times10^{-5}$.  At $\alpha=0.02$ the
difference is approximately $1.3\times10^{-3}$, while at $\alpha=0.05$ the
assembled geometry requires a larger correction,
$s_{\mathrm{2D}}-s_{\mathrm{1D}}\approx0.0164$.  The results show that substituting the cut-cell row coefficient into the
one-dimensional formula is accurate for the smallest triangles but is not a
universal multidimensional formula: the local cut-row coefficient captures the dominant large coefficient in the cut-cell row, while finite geometry and coupling determine the
assembled crossing.

\begin{figure}[ht]
\centering
\includegraphics[width=0.78\textwidth]{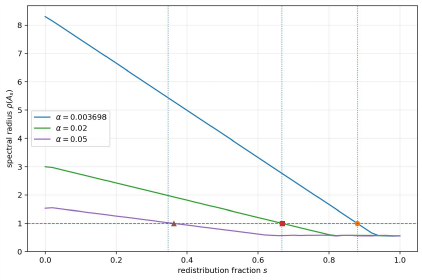}
\caption{Spectral radius of the assembled two-dimensional rotated-channel
operator as the pairwise blending parameter varies.  The horizontal
dashed line is $\rho=1$.  The dotted vertical lines mark the scalar
one-dimensional predictions $s_{\mathrm{1D}}$; the plotted markers identify
the assembled crossings $s_{\mathrm{2D}}$.}
\label{fig:2d-spectral}
\end{figure}

A refinement check for the most restrictive geometry gives
$s_{\mathrm{2D}}=0.87994$ at $N=84$, compared with $0.87947$ at $N=42$,
while $s_{\mathrm{1D}}=0.87955$ on both grids.  This test is intentionally
limited to the eigenvalues of the assembled matrix.  As in the one-dimensional transient study,
$\rho(A_s)\le1$ is not a power-boundedness statement; nonnormal transient
growth can remain near the crossing.  We therefore do not use this experiment
to claim a multidimensional nonlinear or long-time stability theorem.  Its
role is to show, on a genuine cut-cell geometry, that the cut-cell-eigenvalue
the threshold can be computed from the assembled operator and can select a
blending parameter well below full merging.

\subsection{Higher-order accuracy benchmark}

Theorem~\ref{thm:muscl-crossing} analyzes the localized crossing of the
same MUSCL--minmod piecewise-linear minmod region used here, while
Remark~\ref{rem:muscl-periodic-closure} separates that local theorem from
the finite-$N$ periodic closure;
Corollary~\ref{cor:muscl-ssprk2} shows that SSPRK2 preserves that localized-cut-cell mode
crossing.  We now test both the leading value $s_0$ and the finite-$\alpha$
finite-$\alpha$ MUSCL value $s_{\mathrm M}$ in the full periodic benchmark,
and measure their effect on accuracy for
\[
    u_t+u_x=0
\]
on the unit periodic domain.  For unit advection speed, let $T:=1$ denote one
advection period.

The mesh contains $m$ separated cut cells of width $\alpha h$ and $N-m$
regular cells of width $h$ on a unit periodic domain.  Thus
\[
    (N-m)h+m\alpha h=1,
    \qquad
    h=\frac{1}{N-m+m\alpha}.
\]
For the evenly spaced layouts used in the convergence tests, the zero-based
cut-cell indices are
\[
    c_j=\left\lfloor 5+\frac{j}{m-1}(N-11)\right\rfloor,
    \qquad j=0,\ldots,m-1,
\]
for $m>1$ (and $c_0=5$ for $m=1$).  Let $x_i$ denote the center of cell $i$
and $x_{i+1/2}$ its right face.  For two real numbers $r_1$ and $r_2$, define
\[
 \operatorname{minmod}(r_1,r_2)
 =
 \begin{cases}
   \operatorname{sign}(r_1)\min\{|r_1|,|r_2|\}, & r_1r_2>0,\\
   0, & r_1r_2\le0.
 \end{cases}
\]
The piecewise-linear slope and left state at the right face are
\[
 \ell_i=\operatorname{minmod}\!\left(
 \frac{U_i-U_{i-1}}{x_i-x_{i-1}},
 \frac{U_{i+1}-U_i}{x_{i+1}-x_i}
 \right),
 \qquad
 U^-_{i+1/2}=U_i+\ell_i(x_{i+1/2}-x_i).
\]
For positive unit advection speed, the numerical flux is
$F_{i+1/2}=U^-_{i+1/2}$.  If $h_i$ is the width of cell $i$, define the
finite-volume spatial operator
\[
 (\mathcal L_h(U))_i
 =
 -\frac{F_{i+1/2}(U)-F_{i-1/2}(U)}{h_i}.
\]

To make the time discretization and the placement of redistribution explicit,
let $V$ denote an arbitrary stage state and define its provisional forward
Euler update by
\[
    \widehat U(V)=V+\Delta t\,\mathcal L_h(V).
\]
Let $\mathcal R$ denote conservative cell merging: for every cut
cell $c$, $\mathcal R$ replaces $\widehat U_{c-1}$ and $\widehat U_c$ by
their volume-weighted average and leaves all other components unchanged.
The cut cells are separated so these two-cell neighborhoods do not overlap.
For a prescribed scalar $s\in[0,1]$, define one redistributed forward Euler
stage by
\begin{equation}
    \mathcal E_s(V)
    =
    (1-s)\widehat U(V)+s\,\mathcal R(\widehat U(V)).
    \label{eq:Es-benchmark}
\end{equation}
The second-order strong-stability-preserving Runge--Kutta method (SSPRK2) used
in the experiments is then
\begin{equation}
    U^{(1)}=\mathcal E_s(U^n),
    \qquad
    U^{n+1}=\frac12U^n+\frac12\mathcal E_s(U^{(1)}).
    \label{eq:ssprk2}
\end{equation}
Thus redistribution is applied after each forward Euler stage, exactly as in
the numerical implementation.

For comparison, we implement two explicitly defined specializations of
published weighted-SRD constructions on the same separated two-cell
neighborhoods.  First, we specialize the provably stable weighting of Berger
and Giuliani~\cite{BergerGiuliani2024} to this two-cell geometry, with
$\alpha_{\mathrm{target}}=1/2$ and their pre-merging step.  If $l=c-1$ is
the regular cell immediately to the left of a cut cell $c$, the nonzero
weights reduce in this geometry to
\[
    w_{c,c}=1,\qquad
    w_{l,c}=1-2\alpha,\qquad
    w_{l,l}=2\alpha.
\]
Hence
\[
    Q_c^{\mathrm{BG24}}
    =
    \frac{\alpha\widehat U_c+(1-2\alpha)\widehat U_l}{1-\alpha},
\]
with
\[
    (\mathcal R_{\mathrm{BG24}}\widehat U)_c=Q_c^{\mathrm{BG24}},
    \qquad
    (\mathcal R_{\mathrm{BG24}}\widehat U)_l
    =(1-2\alpha)Q_c^{\mathrm{BG24}}+2\alpha\widehat U_l.
\]
The initial cell averages are pre-merged once with
$\mathcal R_{\mathrm{BG24}}$, and the same specialized operator is applied
after each forward Euler stage of SSPRK2.  Numerical comparisons below are
therefore comparisons with this stated specialization, not with every
multidimensional configuration covered by the 2024 algorithm.

Second, we include the $\alpha$--$\beta$ weighting introduced by Giuliani
et al.~\cite{GiulianiEtAl2022}.  Their full multidimensional algorithm
reconstructs neighborhood polynomials before redistributing them.  Because
the present benchmark is a one-dimensional, separated-neighborhood model,
we use the specialization with neighborhood reconstruction slopes set to zero to isolate the published
$\alpha$--$\beta$ weighting itself rather than inventing a one-dimensional
reconstruction stencil not specified by that paper.  With target volume
$V_{\mathrm{target}}=h/2$, the separated two-cell geometry gives
\[
    \beta_c=\frac12-\alpha,\qquad
    \alpha_c=1,\qquad
    \alpha_l=1-\frac{\beta_c}{2},
\]
and
\[
    Q_c^{\alpha\beta}
    =
    \frac{\alpha\widehat U_c+(\beta_c/2)\widehat U_l}
         {\alpha+\beta_c/2}.
\]
The redistributed values are
\[
    (\mathcal R_{\alpha\beta}^{(0)}\widehat U)_c=Q_c^{\alpha\beta},
    \qquad
    (\mathcal R_{\alpha\beta}^{(0)}\widehat U)_l
    =\alpha_l\widehat U_l+\frac{\beta_c}{2}Q_c^{\alpha\beta}.
\]
The superscript $(0)$ records that neighborhood slopes are suppressed.  This
is therefore a controlled specialization of the 2022 weighting, not a claim
to reproduce its full multidimensional production algorithm.  Giuliani
et al.~\cite{GiulianiEtAl2022} themselves note that their weighted variant
can have different stability properties in one-dimensional configurations.

We use the nominal background CFL $\lambda_0=0.4$ and define
\[
    L_0:=\frac{\lambda_0}{\alpha},
    \qquad
    s_0=s_0(L_0)=\frac{L_0-2}{L_0-1}.
\]
For the finite-$\alpha$ MUSCL threshold we also use
\[
    s_{\mathrm M}
    :=
    s_{\mathrm M}^{\mathrm{loc}}(\alpha,L_0)
\]
from~\eqref{eq:sM-local}.  This gives
$s_{\mathrm M}=0.360987\ldots$ for $\alpha=0.15$ and
$s_{\mathrm M}=0.631148\ldots$ for $\alpha=0.10$, compared with
$s_0=0.4$ and $s_0=2/3$, respectively.  As a direct check of the periodic
closure, solving the full one-cut nonlinear eigenvalue equations at $N=150$
gives the same two crossing values to within $1.2\times10^{-15}$.
For a one-period run, let
\[
    n_{\mathrm{step}}:=\left\lceil\frac{T}{\lambda_0 h}\right\rceil,
    \qquad
    \Delta t:=\frac{T}{n_{\mathrm{step}}},
    \qquad
    \lambda_{\mathrm{act}}:=\frac{\Delta t}{h}\le\lambda_0.
\]
Full cell merging uses $s=1$.  Both $s_0$ and $s_{\mathrm M}$ use the nominal value
$L_0$, not the slightly smaller
\[
    L_{\mathrm{act}}:=\frac{\lambda_{\mathrm{act}}}{\alpha}.
\]
This keeps each prescribed blending parameter fixed across grid
refinement rather than allowing the final-step adjustment in $\Delta t$ to
change it.  Over the convergence grids $N\in\{100,200,400,800\}$,
$L_{\mathrm{act}}$ ranges from $2.6638$ to $2.6663$ when
$\alpha=0.15$ ($L_0=2.6667$), and from $3.9912$ to $3.9990$ when
$\alpha=0.10$ ($L_0=4$).  The difference is below $0.9\%$ of $L_0$ in
every reported convergence run.

Evaluating either $s_0(L_0)$ or $s_{\mathrm M}$ requires only a closed-form
scalar expression; no eigenvalue computation, matrix assembly, or iterative
threshold calculation is performed during time stepping.  The redistribution stage is
the same two-cell-merging operation used for $s=1$, followed by the scalar blend in
\eqref{eq:Es-benchmark}.

For the smooth sine-wave test we use
\[
    u(x,0)=\frac12+\frac12\sin(2\pi x).
\]
The Gaussian test uses
\[
    u(x,0)=\exp\!\left[-\frac12
    \left(\frac{d_{\mathbb T}(x,0.25)}{0.06}\right)^2\right],
    \qquad
    d_{\mathbb T}(x,x_0)=\min_{k\in\mathbb Z}|x-x_0+k|.
\]
Reference cell averages are evaluated with 16-point Gauss--Legendre
quadrature on each cell.  Because the exact periodic solution returns to its
initial profile at every integer multiple of $T$, the same quadrature rule
provides the reference averages at $t=kT$.  Let $\overline u_i(t)$ denote this
reference cell average in cell $i$ at time $t$, and define
\begin{equation}
    E_{1,h}(t):=\sum_{i=0}^{N-1} h_i
    \left|U_i(t)-\overline u_i(t)\right|,
    \qquad
    E_{\infty,h}(t):=\max_i
    \left|U_i(t)-\overline u_i(t)\right|.
    \label{eq:benchmark-errors}
\end{equation}
When comparing blending parameters, $E_{1,h}^{(s)}(t)$ and
$E_{\infty,h}^{(s)}(t)$ denote the corresponding errors obtained with blend
$s$.  We also record the discrete mass drift
\begin{equation}
    E_{M,h}(t):=
    \left|\sum_{i=0}^{N-1} h_i U_i(t)
    -\sum_{i=0}^{N-1} h_i U_i(0)\right|.
    \label{eq:mass-drift}
\end{equation}
The one-period convergence tests report these quantities at $t=T$.
Figure~\ref{fig:partial_srd_convergence} shows convergence of
$E_{1,h}(T)$ with ten separated cut cells and
$N\in\{100,200,400,800\}$.  Least-squares slopes of
$\log E_{1,h}(T)$ versus $\log h$ are used for the fitted orders.
For $\alpha=0.15$, the sine-wave orders are $1.87$ for
$s=s_{\mathrm M}$, $1.92$ for the 2022 $\alpha$--$\beta$ zero-slope specialization, $1.94$ for the 2024 provably stable weighted rule, and $1.94$ for full cell merging.
For $\alpha=0.10$, the corresponding orders are $1.91$, $1.92$, $1.94$, and $1.94$.  All four retain approximately second-order convergence.

The finite-$\alpha$ value $s_{\mathrm M}$ also improves slightly on the leading
choice $s_0$ at every measured resolution.  For the sine profile, replacing
$s_0$ by $s_{\mathrm M}$ reduces $E_{1,h}(T)$ by approximately $6.2\%$
at $N=100$ and $4.7\%$ at $N=800$ when $\alpha=0.15$, and by
approximately $3.8\%$ and $3.0\%$, respectively, when $\alpha=0.10$.
Thus the finite-$\alpha$ correction predicted by
Theorem~\ref{thm:muscl-crossing} is visible but smaller than the primary
gain obtained by moving away from full cell merging.

\begin{figure}[ht]
\centering
\includegraphics[width=0.86\textwidth]{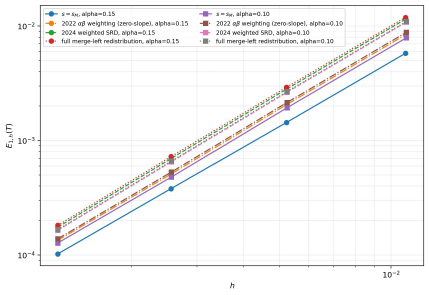}
\caption{Grid convergence of $E_{1,h}(T)$ for the smooth periodic sine-wave
benchmark using the MUSCL value from Theorem~\ref{thm:muscl-crossing} $s=s_{\mathrm M}$, the
specialization with neighborhood reconstruction slopes set to zero of the 2022 $\alpha$--$\beta$
weighting~\cite{GiulianiEtAl2022}, the 2024 provably stable weighted
SRD rule~\cite{BergerGiuliani2024}, and full cell merging.
All four retain approximately second-order convergence.}
\label{fig:partial_srd_convergence}
\end{figure}

Table~\ref{tab:raw-convergence} reports the directly measured values for
$s=s_{\mathrm M}$ and the three comparators.  The complete $s_0$ and
$s_{\mathrm M}$ measurements are provided in
\texttt{muscl\_s0\_vs\_sM\_convergence.csv}.

\begin{table}[ht]
\centering
\small
\begin{tabular}{ccrrrrr}
$\alpha$ & profile & $N$ &
$E_{1,h}^{(s_{\mathrm M})}(T)$ &
$E_{1,h}^{(\alpha\beta,0)}(T)$ &
$E_{1,h}^{(\mathrm{BG24})}(T)$ &
$E_{1,h}^{(\mathrm{merge})}(T)$ \\
\hline
0.15 & sine & 100 & $5.782e-03$ & $8.399e-03$ & $1.150e-02$ & $1.192e-02$ \\
0.15 & sine & 200 & $1.439e-03$ & $2.060e-03$ & $2.812e-03$ & $2.923e-03$ \\
0.15 & sine & 400 & $3.802e-04$ & $5.103e-04$ & $6.964e-04$ & $7.241e-04$ \\
0.15 & sine & 800 & $1.023e-04$ & $1.344e-04$ & $1.750e-04$ & $1.816e-04$ \\
0.15 & Gaussian & 100 & $2.784e-02$ & $3.265e-02$ & $3.805e-02$ & $3.875e-02$ \\
0.15 & Gaussian & 200 & $8.346e-03$ & $9.627e-03$ & $1.135e-02$ & $1.160e-02$ \\
0.15 & Gaussian & 400 & $2.580e-03$ & $2.787e-03$ & $3.072e-03$ & $3.120e-03$ \\
0.15 & Gaussian & 800 & $7.025e-04$ & $7.434e-04$ & $8.025e-04$ & $8.120e-04$ \\
0.10 & sine & 100 & $7.899e-03$ & $8.820e-03$ & $1.085e-02$ & $1.098e-02$ \\
0.10 & sine & 200 & $1.931e-03$ & $2.149e-03$ & $2.640e-03$ & $2.680e-03$ \\
0.10 & sine & 400 & $4.795e-04$ & $5.305e-04$ & $6.521e-04$ & $6.624e-04$ \\
0.10 & sine & 800 & $1.274e-04$ & $1.386e-04$ & $1.648e-04$ & $1.671e-04$ \\
0.10 & Gaussian & 100 & $3.192e-02$ & $3.361e-02$ & $3.717e-02$ & $3.740e-02$ \\
0.10 & Gaussian & 200 & $9.382e-03$ & $9.859e-03$ & $1.098e-02$ & $1.107e-02$ \\
0.10 & Gaussian & 400 & $2.751e-03$ & $2.823e-03$ & $3.004e-03$ & $3.023e-03$ \\
0.10 & Gaussian & 800 & $7.352e-04$ & $7.502e-04$ & $7.887e-04$ & $7.920e-04$ \\
\end{tabular}
\caption{Measured one-period $E_{1,h}(T)$ values for the smooth convergence
tests using the finite-$\alpha$ MUSCL threshold and the three comparison operators.
All entries are direct simulation outputs.}
\label{tab:raw-convergence}
\end{table}
\subsection{Interpolated cell count across fixed error levels}

For the sine-wave data, we compare the total cell count required to reach
five prescribed error levels,
\[
    E_{1,h}(T)\in
    \{5\times10^{-3},\,2\times10^{-3},\,10^{-3},\,
      5\times10^{-4},\,2\times10^{-4}\}.
\]
Interpolation is performed in $\log h$ versus $\log E_{1,h}(T)$ between the
two measured resolutions that bracket each target.  The resulting mesh scale
$h_*$ is converted back to total cell count through
\[
    N_{\mathrm{est}}=\frac{1}{h_*}+m(1-\alpha),
    \qquad m=10.
\]
Every reported value lies between measured resolutions
$N\in\{100,200,400,800\}$; no extrapolation is used.

At $E_{1,h}(T)=10^{-3}$ and $\alpha=0.15$, the estimates are
$N_{\mathrm{est}}=241$ for $s=s_{\mathrm M}$, $248$ for $s=s_0$,
$286$ for the 2022 $\alpha$--$\beta$ zero-slope specialization, $334$ for
the 2024 provably stable weighted rule, and $340$ for full cell merging
redistribution.  The finite-$\alpha$ MUSCL value therefore reduces the required cell
count by $2.7\%$ relative to $s_0$, $15.6\%$ relative to the 2022
specialization, and $27.7\%$ relative to the 2024 weighted rule.

For $\alpha=0.10$, the corresponding estimates are $277$, $282$, $292$,
$323$, and $326$.  The reductions for $s=s_{\mathrm M}$ are $1.8\%$,
$5.0\%$, and $14.3\%$ relative to $s_0$, the 2022 specialization, and the
2024 weighted rule, respectively.

Across all five target errors, the reduction of $s=s_{\mathrm M}$ relative
to the 2024 weighted rule ranges from $25.0\%$ to $28.5\%$ for
$\alpha=0.15$ and from $12.9\%$ to $14.3\%$ for $\alpha=0.10$.
Figure~\ref{fig:partial_srd_required_grid_cells} shows the theorem-based
MUSCL value against the three comparison operators.  These are interpolated
cell-count estimates at fixed error, not wall-clock timing measurements.

\begin{figure}[ht]
\centering
\includegraphics[width=0.88\textwidth]{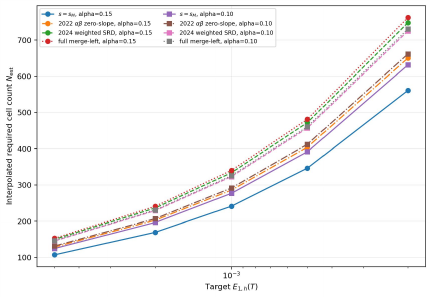}
\caption{Interpolated total cell count $N_{\mathrm{est}}$ required to reach
five fixed values of $E_{1,h}(T)$ in the sine-wave benchmark.  The
MUSCL value from Theorem~\ref{thm:muscl-crossing} $s=s_{\mathrm M}$ is compared with the 2022
$\alpha$--$\beta$ zero-slope specialization~\cite{GiulianiEtAl2022}, the
2024 provably stable weighted SRD rule~\cite{BergerGiuliani2024}, and full
cell merging.  Interpolation is performed in $\log h$ versus
$\log E_{1,h}(T)$, and every target is bracketed by measured resolutions.}
\label{fig:partial_srd_required_grid_cells}
\end{figure}

\subsection{Dependence on cut-cell count and layout}

The preceding test uses ten approximately evenly separated cut cells.  We
therefore repeat the one-period smooth tests on $N=400$ cells while varying
the number of non-adjacent cut cells over
$m\in\{1,2,5,10,20\}$.  For each $m$ we use one evenly spaced layout and ten pseudorandom non-adjacent layouts.
The pseudorandom layouts are generated with NumPy's PCG64 generator using seed
$20260808$: indices are sampled without replacement from
$\{5,\ldots,N-6\}$ and a draw is rejected if any two selected cut cells are
adjacent.  The file \texttt{partial\_srd\_layout\_robustness\_seeded\_raw.csv}
records the cut indices for every accepted layout.  All other numerical choices are unchanged.
For each layout define the relative one-period error reduction
\begin{equation}
    \mathcal G_1
    :=100\,\frac{E_{1,h}^{(1)}(T)-E_{1,h}^{(s_0)}(T)}
                   {E_{1,h}^{(1)}(T)}.
    \label{eq:calG1}
\end{equation}
Thus $\mathcal G_1>0$ is equivalent to
$E_{1,h}^{(s_0)}(T)<E_{1,h}^{(1)}(T)$.

Every $s=s_0$ and $s=1$ run reaches $T=1$, and
$E_{M,h}(T)<1.7\times10^{-15}$ in all seeded-layout runs.  The accuracy
difference grows with the number of redistributed neighborhoods.  For a
Gaussian profile with $\alpha=0.15$, the mean $\mathcal G_1$ is $0.7\%$, $1.6\%$, $7.0\%$, $16.8\%$, and $36.2\%$ as $m$
increases from $1$ to $20$.  For the smooth sine wave, the corresponding
means are $-0.7\%$, $0.9\%$, $21.0\%$, $44.6\%$, and $52.1\%$.  At
$\alpha=0.10$ the same qualitative trend holds with smaller gains because
$s_0=2/3$ is closer to full cell merging.  Once $m\ge5$, $\mathcal G_1>0$
for every tested layout for both smooth profiles.

\paragraph{Practical guidance.}
For $m=1$ or $2$, the measured gain is small and can change sign.  Within this
benchmark, there is therefore little accuracy-based reason to replace full
redistribution when only one or two isolated cut cells are active.  The case
for $s_0$ becomes stronger when several separated redistribution
neighborhoods are present and the accumulated dissipation from full
redistribution materially affects the error.

\begin{figure}[ht]
\centering
\includegraphics[width=0.48\textwidth]{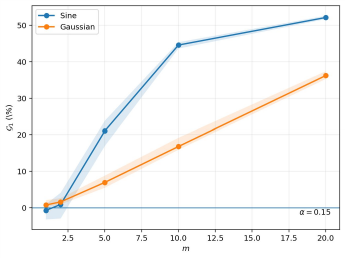}
\hfill
\includegraphics[width=0.48\textwidth]{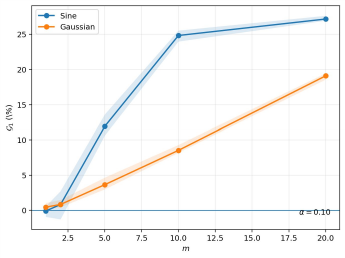}
\caption{$\mathcal G_1$ from \eqref{eq:calG1} versus the number $m$ of separated cut
cells.  Each point is the mean over one evenly spaced and ten pseudorandom
layouts; shaded bands show the observed range.  \textit{Left}:
$\alpha=0.15$.  \textit{Right}: $\alpha=0.10$.  Seed $20260808$ and every
accepted cut-cell index set are recorded with the raw data.}
\label{fig:partial_srd_layout_robustness}
\end{figure}

\subsection{100-period higher-order test}

We extend the smooth benchmark to
$100$ advection periods with $N=200$ and ten evenly spaced cut cells.  For
$k=0,\ldots,100$, the error is sampled at $t=kT$.  The same
nominal background CFL $\lambda_0=0.4$ is used.  We test both the leading
choice $s_0$ and the finite-$\alpha$ MUSCL value $s_{\mathrm M}$; full cell merging
redistribution uses $s=1$.

All eight $s=s_0$ and $s=1$ runs complete 100 periods, and
$E_{M,h}(t)<8\times10^{-14}$ at every recorded integer period.  The four
additional $s=s_{\mathrm M}$ runs also complete 100 periods.  Their
$E_{1,h}(100T)$ values are $0.117$ and $0.149$ for the sine and Gaussian
profiles at $\alpha=0.15$, and $0.144$ and $0.158$ at $\alpha=0.10$.
The largest recorded integer-period mass drift among these finite-$\alpha$-threshold
runs is $1.21\times10^{-12}$.  At $\alpha=0.10$,
$E_{1,h}(100T)$ is $0.148$ for $s=s_0$ and $0.178$ for $s=1$ in the sine
case, and $0.159$ and $0.168$, respectively, in the Gaussian case.  At
$\alpha=0.15$, the corresponding values are $0.123$ and $0.187$ for sine,
and $0.151$ and $0.171$ for Gaussian.  These are finite-duration numerical observations and are interpreted only as
empirical boundedness over the tested interval.  For the
control runs, failure is declared when a non-finite component occurs or when
$\|U^n\|_\infty>10^8$.  Under this criterion, the same higher-order scheme
with $s=0$ fails before $0.06T$ for both
$\alpha=0.10$ profiles at $N=200$, whereas both $\alpha=0.15$ control runs
complete 100 periods.  Table~\ref{tab:long-time} collects the corresponding
$s_{\mathrm M}$, $s_0$, and full-merging measurements.

\begin{table}[ht]
\centering
\small
\begin{tabular}{cccrr}
$\alpha$ & profile & $s$ & $E_{1,h}(100T)$ & $\max_{0\le k\le100}E_{M,h}(kT)$ \\
\hline
0.10 & sine     & $s_{\mathrm M}$ & 0.144 & $4.92\times10^{-14}$ \\
0.10 & sine     & $s_0$            & 0.148 & $4.85\times10^{-14}$ \\
0.10 & sine     & $1$              & 0.178 & $7.74\times10^{-14}$ \\
0.10 & Gaussian & $s_{\mathrm M}$ & 0.158 & $1.46\times10^{-14}$ \\
0.10 & Gaussian & $s_0$            & 0.159 & $1.47\times10^{-14}$ \\
0.10 & Gaussian & $1$              & 0.168 & $2.35\times10^{-14}$ \\
0.15 & sine     & $s_{\mathrm M}$ & 0.117 & $1.20\times10^{-12}$ \\
0.15 & sine     & $s_0$            & 0.123 & $1.11\times10^{-14}$ \\
0.15 & sine     & $1$              & 0.187 & $2.53\times10^{-14}$ \\
0.15 & Gaussian & $s_{\mathrm M}$ & 0.149 & $3.24\times10^{-13}$ \\
0.15 & Gaussian & $s_0$            & 0.151 & $3.00\times10^{-15}$ \\
0.15 & Gaussian & $1$              & 0.171 & $7.61\times10^{-15}$ \\
\end{tabular}
\caption{One-hundred-period higher-order measurements for $N=200$ and ten
evenly spaced cut cells.  Every listed run reaches $100T$.  The last column is
the maximum recorded mass drift over the integer-period samples.}
\label{tab:long-time}
\end{table}

\begin{figure}[ht]
\centering
\includegraphics[width=0.48\textwidth]{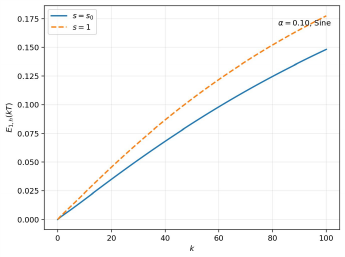}
\hfill
\includegraphics[width=0.48\textwidth]{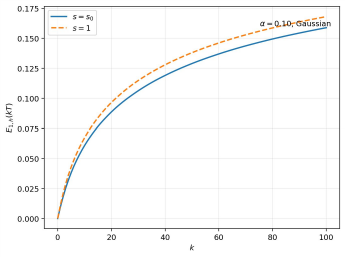}\\[4pt]
\includegraphics[width=0.48\textwidth]{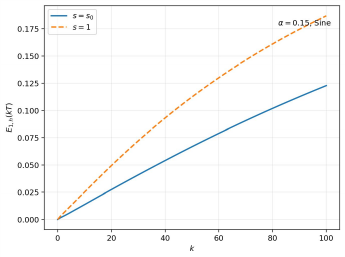}
\hfill
\includegraphics[width=0.48\textwidth]{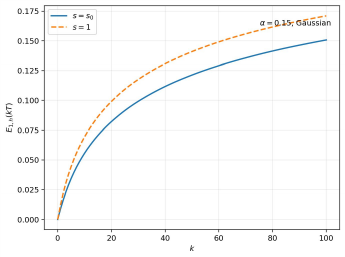}
\caption{Measured $E_{1,h}(kT)$, $k=0,\ldots,100$, for
the higher-order benchmark with $N=200$ and ten evenly spaced cut cells.
The plotted $s=s_0$ and $s=1$ runs complete each of the four combinations
of profile and $\alpha$, with $E_{M,h}<8\times10^{-14}$ at the recorded
integer periods.  The additional $s=s_{\mathrm M}$ measurements are listed
in Table~\ref{tab:long-time}.  Left column: sine; right column: Gaussian.
Top row: $\alpha=0.10$; bottom row: $\alpha=0.15$.}
\label{fig:long-time-higher-order}
\end{figure}

The first-order matrix theorem, Theorem~\ref{thm:criterion}, does not by
itself cover these higher-order experiments.  The localized MUSCL crossing is analyzed directly in
Theorem~\ref{thm:muscl-crossing} and Corollary~\ref{cor:muscl-ssprk2}.
The exact finite-$N$ periodic crossing, the multi-cut accuracy results, and
the 100-period statements are finite-resolution numerical evidence.

\subsection*{Scope of the numerical evidence}
Theorem~\ref{thm:criterion} gives the exact first-order spectral crossing,
while Theorem~\ref{thm:muscl-crossing} and
Corollary~\ref{cor:muscl-ssprk2} derive the localized crossing for its
MUSCL--minmod branch analyzed in Theorem~\ref{thm:muscl-crossing}.  The full periodic one-cut
crossing is checked numerically rather than promoted to a theorem.
The numerical experiments use finite $\alpha$ and multiple separated cut
cells, so the measured multi-cut accuracy gains remain empirical at finite
resolution.  Section~\ref{sec:2d-experiment} extends the spectral experiment
to an assembled two-dimensional embedded-boundary operator, but it does not
turn the scalar one-dimensional formula into a multidimensional theorem.  The
MUSCL theorem is a localized-cut-cell mode result and does not establish global
nonlinear stability, power boundedness, or a universal threshold for PPM,
WENO, implicit integrators, adjacent cut cells, nonlinear systems, or
multidimensional weighted-SRD neighborhoods.  Those cases require
analyzing the corresponding discrete operator and how redistribution changes
its unstable eigenvalues.

\subsection*{Reproducibility summary}
Table~\ref{tab:reproducibility} consolidates the numerical choices used in the
higher-order experiments.

\begin{table}[ht]
\centering
\small
\begin{tabular}{ll}
quantity & value or definition \\
\hline
domain and speed & $x\in[0,1)$ periodic, $a=1$, $T=1$ \\
mesh scale & $h=(N-m+m\alpha)^{-1}$ \\
cut-cell width & $\alpha h$ \\
reconstruction & piecewise linear with minmod limiter \\
time integrator & SSPRK2, redistribution after each Euler stage \\
nominal background CFL & $\lambda_0=0.4$ \\
partial blends & $s=s_0(L_0)$ or $s=s_{\mathrm M}(\alpha,L_0)$ \\
reference averages & 16-point Gauss--Legendre quadrature per cell \\
Gaussian width & $0.06$ in $d_{\mathbb T}(x,0.25)$ \\
layout seed & NumPy PCG64, seed $20260808$ \\
control failure criterion & non-finite state or $\|U^n\|_\infty>10^8$ \\
\end{tabular}
\caption{Numerical settings used in the higher-order benchmark.  The
convergence section gives the deterministic cut-cell placement formula, and
the seeded-layout data file records every pseudorandom accepted index set.}
\label{tab:reproducibility}
\end{table}

The accompanying raw-data files are
\texttt{partial\_srd\_grid\_convergence.csv},
\texttt{partial\_srd\_layout\_robustness\_seeded\_raw.csv},
\texttt{verified\_threshold\_check.csv},
\texttt{verified\_transient\_time\_history\_summary.csv},
\texttt{verified\_transient\_time\_history\_summary\_alpha\_0.15.csv},
\texttt{long\_time\_N200\_alpha\_010\_history\_notation.csv},
\texttt{long\_time\_N200\_alpha\_015\_history\_notation.csv},
\texttt{four\_method\_convergence.csv},
\texttt{muscl\_s0\_vs\_sM\_convergence.csv},
\texttt{muscl\_sM\_fixed\_error\_comparison.csv},
\texttt{muscl\_sM\_long100\_integer\_mass.csv},
\texttt{muscl\_threshold\_crossing\_check.csv},
\texttt{two\_dim\_spectral\_summary.csv},
\texttt{two\_dim\_spectral\_curves.csv}, and
\texttt{two\_dim\_refinement\_check.csv}.  The seeded
layout file records every accepted cut-cell index set.

\clearpage

\section{Conclusion}
\label{sec:conclusion}

The periodic upwind model with cell merging yields an explicit spectral
threshold.  For sufficiently small $\alpha$, the base update has a
localized eigenvalue outside $\overline{\mathbb D}$ whose eigenvector is
concentrated at the cut cell.  The cell-merging correction approaches the
same direction in the volume-weighted geometry.  Tracking the localized branch
$\mu_u(s)$ therefore determines the redistribution threshold.

For the full $N$-cell blended operator
\[
    A_s=(1-s)A+sA_{\mathcal R},
\]
fixed $L>2$, and $N\ge4$, the localized eigenvalue re-enters the closed
unit disk through $\mu=-1$ at
\[
 s_{-1}^{(N)}(\alpha,L)
 =
 \frac{L-2}{L-1}
 -
 \frac{L(L-2)}{2(L-1)^2}\alpha^2
 +O(\alpha^3).
\]
Thus
\[
    s_0(L)=\frac{L-2}{L-1}
\]
is asymptotically sufficient and differs from the exact local redistribution
threshold by $O(\alpha^2)$ in the regime of the theorem.  The quantity $s_0(L)$ is the leading-order blending parameter associated
with the return of $\mu_u$ to the unit disk.  For moderate fixed $L>2$, $s_0(L)$ can be substantially below $1$, while
$s_0(L)\to1$ as $L\to\infty$.

The theorem is spectral; uniform power boundedness is a separate property.
Direct finite-time calculations in the volume-weighted induced norm show transient
amplification of approximately $2.01$ for $\alpha=0.10$ and $2.34$ for
$\alpha=0.15$ at the exact $L=4$ crossing, even though no eigenvalue lies
outside the unit disk.  Thus the redistribution threshold eliminates exponential growth of the localized
mode, while nonnormal transient growth can remain.

The broader block-matrix analysis identifies what is structural and what is
specific to the model.  The matrix $\Gamma$ controls the unbounded
spectrum of a class of singularly perturbed finite-volume operators, and the
associated invariant subspace converges to the subspace spanned by the cut-cell coordinates.
The existence and localization of these cut-cell modes therefore generalize
more readily than the scalar threshold formula.  For a different flux,
reconstruction, nonlinear system, or multidimensional redistribution rule,
the corresponding leading cut-cell block and the action of redistribution on its
spectrum must be derived again.

The MUSCL analysis applies the same localized-eigenvalue calculation to the
second-order minmod stage used in the benchmark and produces both the leading
value $s_0$ and a computable
finite-$\alpha$ correction $s_{\mathrm M}$.  In the higher-order benchmark,
$s=s_{\mathrm M}$ retains approximately second-order convergence and has the
lowest error among the tested operators.  At $E_{1,h}(T)=10^{-3}$, it
requires approximately $27.7\%$ fewer cells than the 2024 provably stable
weighted rule for $\alpha=0.15$ and $14.3\%$ fewer for $\alpha=0.10$.
The corresponding reductions relative to the 2022 $\alpha$--$\beta$
zero-slope specialization are $15.6\%$ and $5.0\%$.  The 2022 comparison is
intentionally labeled as a zero-slope specialization and is not presented as
a reproduction of the full multidimensional weighted algorithm.

The rotated-channel experiment gives a complementary multidimensional check.
For the assembled two-dimensional cut-cell matrix, the unstable eigenvectors
remain concentrated near the boundary cut cells.  The smallest redistribution
fraction for which $\rho(A_s)\le1$ is substantially below full cell merging.
Substituting the cut-cell row coefficient into the one-dimensional formula
nearly matches this threshold for the smallest triangles but underpredicts it
by about $0.0164$ at $\alpha=0.05$.  Thus the localization persists, while the
finite-dimensional threshold depends on the assembled geometry and the
redistribution operator.

In separate 100-period tests with $N=200$, all eight $s=s_0$ and $s=1$
runs for the sine and Gaussian profiles complete the integration, and
$E_{M,h}(t)<8\times10^{-14}$ at every recorded integer period.  These observations show bounded evolution over the tested 100 periods for
$s=s_0$ and $s=s_{\mathrm M}$.  They do not establish uniform stability.

The present result is a model theorem.  A next extension is to identify the
eigenvalue cluster localized near the cut cells of a multidimensional
weighted-SRD neighborhood.  If that cluster is denoted by $\Sigma_c(A_s)$,
one could determine
\[
    s_*=\inf\left\{s:
    \max_{\mu\in\Sigma_c(A_s)}|\mu|\le1\right\}.
\]
This would choose the blending parameter from the localized cut-cell
eigenvalues instead of prescribing full cell merging independently of those
eigenvalues.

\section*{Data availability}
All numerical data and scripts used to generate the reported computational
results are included in the supplementary material accompanying this
submission.

\section*{Declarations}

\noindent\textbf{Funding.} This research did not receive any specific grant from funding agencies in the public, commercial, or not-for-profit sectors.

\noindent\textbf{Declaration of competing interests.}
Declarations of interest: none.

\noindent\textbf{CRediT author statement.} Justo E. Karell: Conceptualization,
Methodology, Formal analysis, Software, Validation, Investigation, Data
curation, Visualization, Writing -- original draft, Writing -- review \&
editing.

\section*{Declaration of generative AI and AI-assisted technologies in the manuscript preparation process}
During the preparation of this work, the author used ChatGPT (OpenAI) to
assist with numerical code generation, consistency checking, and language
editing.  The author reviewed and verified the resulting analysis, code,
numerical outputs, and manuscript text and takes full responsibility for the
content of the article.

\end{document}